\documentclass{amsart}
\newtheorem{mtheorem}{Main Theorem}[section]
\newtheorem{theorem}{Theorem}[section]
\newtheorem{lemma}[theorem]{Lemma}
\newtheorem{cor}[theorem]{Corollary}
\theoremstyle{definition}
\newtheorem{definition}[theorem]{Definition}
\newtheorem{pro}[theorem]{Proposition}

\theoremstyle{remark}
\newtheorem{remark}[theorem]{Remark}

\numberwithin{equation}{section}

\begin{document}

\title[AN ANCIENT SOLUTION OF THE RICCI FLOW IN DIMENSION 4]{AN ANCIENT SOLUTION OF THE RICCI FLOW IN DIMENSION 4 CONVERGING TO THE EUCLIDEAN SCHWARZSCHILD METRIC}

\author{Ryosuke Takahashi}
\address{Department of Mathematics, Harvard University, Cambridge, Massachusetts 02138}
\email{tryo@math.harvard.edu}



\date{}

\keywords{Euclidean Schwarzschild metrics, Ricci flows, ancient solutions.}

\begin{abstract}
In this paper, we prove the existence of an ancient solution to the Ricci flow whose limit at $t=-\infty$ is the Euclidean Schwarzschild metric.
\end{abstract}

\maketitle

\section{Introduction}

 In this paper, we consider the Euclidean Schwarzschild metric which is defined on $M=[0, 4\pi]\times (1,\infty)\times \mathbf{S}^2$:
\[
g_0=(1-r^{-1})dt^2+(1-r^{-1})^{-1}dr^2+r^2d\Omega^2.
\]
This metric is obtained from the Schwarzschild metric [10] by changing the time coordinate into $\tau=-it$ and taking black-hole mass to be $\frac{1}{2}$. We also have $t$ is periodic with period $4\pi$. Therefore, this manifold is diffeomorphic to $\mathbf{S}^1\times \mathbb{R}^1\times \mathbf{S}^2$ and its metric is asymptotic to the flat metric on $\mathbf{S}^1\times\mathbb{R}^3$ as $r\rightarrow \infty$, see [4], [7].\\

Since $g_0$ is a Ricci flat metric (cf. section 2), it is a solution of the vacuum Einstein equation $R_{ij}-\frac{1}{2}g_{ij}R=0$ and $g_0$ can be regarded as a critial point of the Einstein-Hilbert action 
\[
S(g)=-\frac{1}{2\kappa}\int R\sqrt{g}dtdrd\Omega,
\]
where $R$ is the scalar curvature of $g$ and $\kappa$ is a constant. This metric is unstable, i.e. the bilinear form
\[
B[h,k]=\frac{\partial}{\partial s}\frac{\partial}{\partial t}S(g_0+sh+tk)|_{s=t=0}
\]
which is defined on $C_{0}^{\infty}(Sym^2(T^*M))\times C_{0}^{\infty}(Sym^2(T^*M))$ is not positive semi-definite. Here $C_{0}^{\infty}(Sym^2(T^*M))$ is the set of symmetric $C^{\infty}$ 2-tensors $|h|$ such that $|h|\rightarrow 0$ uniformly as $r\rightarrow 1$ and $r\rightarrow \infty$. This means there exists $h$ such that $B[h,h]<0$.\\

 In [4], Allen argued the instability of $g_0$ by showing that there exist a negative eigenvalue $\lambda$ and a corresponding eigenvector $h \in C_{0}^{\infty}(Sym^2(T^*M))$ which satisfy
\begin{align}\label{ee}
\frac{\partial}{\partial \delta}(2Ric(g_0+\delta h))|_{\delta=0}=\lambda h.
\end{align}
This implies that $B[h,h]<0$ by computing second variation of $S$ directly. However, [4] doesn't provide a rigorous proof to show the existence of $h$, so we shall give a proof in this paper.\\

The existence of $h$ tells us how to evolve the Euclidean Schwarzschild metric to an ancient solution of Ricci flow. Roughly speaking, by equation (\ref{ee}), we have $-2Ric(g_0+\delta h)= -\lambda \delta h + o(\delta)$. So if we take $\delta=e^{-\lambda t}$, we will have
\[
\frac{\partial}{\partial t}(g_0+\delta h)=-\lambda\delta h = -2Ric(g_0+\delta h)+o(\delta).
\]
This means that $(g_0+\delta h)$ gives us the first order approximation of a Ricci flow. \\

The main purpose of this paper is to find an ancient solution of Ricci flow $g$ with $g(x,-\infty)=g_0$ and $g(x,t)$ will flow along the direction given by $h$ when $t$ tends to $-\infty$.

\begin{mtheorem}
There exist $N\in \mathbb{R}$ and an ancient solution $g$ satisfying
\begin{align*}
\frac{\partial}{\partial t}g(x,t)&=-2Ric(x,t)\mbox{ for all } x\in M \mbox{ and }t\in (-\infty,N],\\
 g(x,t)&\rightarrow g_0(x)\mbox{ uniformly as }t\rightarrow-\infty.
\end{align*}
\end{mtheorem}

\noindent {\bf Acknowledgement:} The author wants to thank professor Chang-Shou Lin and professor Chiun-Chuan Chen for their encouragement during past four years. He also wants to thank his colleagues: Chen-Yu Chi, Jie Zhou, Yi-Li and Yu-Shen Lin for their help on many aspects. Finally, the author wants to thank his advisor Clifford Taubes, who gives him so much help.\\ 

\section{Basic Computations:}
Let $(M,g_0)$ be the Euclidean Schwarzschild metric. In this section, we compute its Christoffel symbols and curvature tensors. We have:
\[
\Gamma^1_{00}=-\frac{1}{2}(1-r^{-1})\frac{1}{r^2}=\frac{-(1-r^{-1})}{2r^2}, \mbox{ }\Gamma^{0}_{01}=\frac{1}{2}(1-r^{-1})^{-1}\frac{1}{r^2}=\frac{1}{2(1-r^{-1})r^2}
\]
\[
\Gamma^1_{11}=\frac{-1}{2(1-r^{-1})r^2}, \mbox{ }\Gamma^2_{12}=\frac{1}{r}, \mbox{ }\Gamma^3_{13}=\frac{1}{r}
\]
\[
\Gamma^1_{22}=\frac{-1}{2}(1-r^{-1})(2r)=-(1-r^{-1})r, \mbox{ }\Gamma^3_{23}=\cot\theta
\]
\[
\Gamma^1_{33}=-(1-r^{-1})r\sin^2\theta, \mbox{ }\Gamma^2_{33}=-\sin \theta\cos \theta
\]
and
\[
R_{0101}=r^{-3}, \mbox{ }R_{0202}=-\frac{1}{2}(1-r^{-1})r^{-1}
\]
\[
R_{0303}=-\frac{1}{2}(1-r^{-1})r^{-1}\sin^2\theta, \mbox{ }R_{1212}=-\frac{1}{2(1-r^{-1})r}
\]
\[
R_{1313}=-\frac{1}{2(1-r^{-1})r}\sin^2\theta, \mbox{ }R_{2323}=r\sin^2\theta
\]
and
\[
R_{ij}=0\mbox{ }\mbox{ }\mbox{ }\forall i,j=0,1,2,3.
\]

It is well known that this metric has a bounded Riemannian curvature, that is
\[
|\sum_{\alpha,\beta,\gamma,\delta,\lambda,\mu,\nu,\xi}g^{\alpha\beta}g^{\gamma\delta}g^{\lambda\mu}g^{\nu\xi}R_{\alpha\gamma\lambda\nu}R_{\beta\delta\mu\xi}|=|Rm|^2<C
\]
Here we have: $\sum_{\alpha,\beta,\gamma,\delta,\lambda,\mu,\nu,\xi}g^{\alpha\beta}g^{\gamma\delta}g^{\lambda\mu}g^{\nu\xi}R_{\alpha\gamma\lambda\nu}R_{\beta\delta\mu\xi}= \sum_{i,j}g^{ii}g^{ii}g^{jj}g^{jj}R_{ijij}R_{ijij}$ (We use summation convention in all following paragraphs). Actually, we have a stronger conclusion: Observing that $|R_{ijij}|=Cr^{-3}g_{ii}g_{jj}$, so all sectional curvatures satisfy the inequality $|K_{ij}|\leq r^{-3}$. Moreover, if $h= h_{00} dt^2+h_{11}dr^2+ h_{22} d\theta^2 +h_{33}d\varphi^2$ is a 2-tensor, we have
\begin{align}\label{ie1}
R^{ijij}h_{ii}h_{jj} \leq \frac{1}{2}[r^{-3}(g^{ii})^2h_{ii}^2+r^{-3}(g^{jj})^2h_{jj}^2]\leq r^{-3}|h|^2.
\end{align}

\begin{remark} Sometime we change our coordinate by taking $p^2=(1-r^{-1})$, $p\in(0,1)$. In the new coordinate
\[
g_0=p^2dt^2+\frac{1}{4}\frac{1}{(1-p^2)^2}dp^2+\frac{1}{(1-p^2)^2}d\Omega^2.
\]  
This coordinate has several advantages. First of all, $g_{\alpha\beta}$ are finite near $p=0$ locus. Secondly, we have our $M$ is a punctured disk $\mathbf{S}^1\times (0,1)$ times a $\mathbf{S}^2$ now (That is, $r=1$ locus will be a point). To avoid confusion, we call this coordinate the $p$-coordinate.
\end{remark}

\section{Eigenvectors with Negative Eigenvalues: }
In this section, we prove the existence of solution of a modified equation (equation (\ref{mee})). 

\subsection{Ricci-de Turck flows and modified eigenvectors: } Recall that we want to prove the existence of $h$ which satisfies
\[
\frac{\partial}{\partial \delta}(-2Ric(g_0+\delta h))|_{\delta=0}=-\lambda h.
\]
The left hand side of this equation can be expressed as
\begin{align}\label{ee1}
(\Delta_Lh)_{kl}+\nabla_{k}\nabla_{l}H+\nabla_{k}(\zeta h)_l+\nabla_{l}(\zeta h)_k.
\end{align}
where $H$ is the trace of $h$, $(\zeta h)_k$ is defined by $-g^{ij}\nabla_ih_{jk}$ and the $\mathbf{Lichnerowicz}$  $\mathbf{Laplacian}$ $\Delta_L$ is defined by
\[
(\Delta_Lh)_{jk}=\Delta h_{jk}+2R^{r\mbox{ }p}_{\mbox{ }j\mbox{ }k}h_{rp}-g^{pq}R_{jp}h_{qk}-g^{pq}R_{kp}h_{jq},
\]
see [3]. So our equation becomes
\begin{align}\label{ee2}
(\Delta_Lh)_{kl}+\nabla_{k}\nabla_{l}H+\nabla_{k}(\zeta h)_l+\nabla_{l}(\zeta h)_k=-\lambda h_{kl}.
\end{align}

Ideally,  we want to use variational approach to obtain a solution $h$. This means we need to find a functional
\[
\mathbf{\hat{a}}:C_{0}^{\infty}(Sym^2(T^*M))\rightarrow \mathbb{R}
\]
such that the 
minimizer of $\mathbf{\hat{a}}$ over $C_{0}^{\infty}(Sym^2(T^*M))\cap\{\|h\|=1\}$ will satisfy the Euler-Lagrange equation $(\Delta_Lh)_{kl}+\nabla_{k}\nabla_{l}H+\nabla_{k}(\zeta h)_l+\nabla_{l}(\zeta h)_k=-\lambda h_{kl}$. Unfortunately, we don't know how to find this $\mathbf{\hat{a}}$.\\

However, if we consider a simpler equation
\begin{align}\label{mee}
(\Delta_Lh)_{kl}=-\lambda h_{kl},
\end{align}
we know how to define a functional $\mathbf{a}$ whose Euler-Lagrange equation is (\ref{mee}) (the formula of $\mathbf{a}$ will appear in the next subsection). Moreover, equations (\ref{mee}) and (\ref{ee2}) are related. The relation between these two equations can be explained by de Turck's argument [5]. In following paragraphs, we show how to connect these two equation by de Turck's method.\\

Here we follow notations of section 1. We claim the following statement: suppose that we have an ancient solution of Ricci flow $\hat{g}(x,t)$, which is defined on $M\times (-\infty, N]$ for some $N\in \mathbb{R}$, such that after changing the time variable by taking $t=\frac{1}{-\lambda}\log(\delta)$ we have $\hat{g}(x,\delta)=g_0+\delta h+o(\delta)$. Then there exist a family of diffeomorphisms $\varphi_\delta: M \rightarrow M$, $\varphi_{0}=id$ such that if we write $g=\varphi_\delta^*(\hat{g})=g_0+\delta \bar{h} +o(\delta)$, $\bar{h}$ will satisfy equation (\ref{mee}). Conversely, if we have a $\bar{h}$ satisfying (\ref{mee}) and a $g=\varphi_\delta^*(\hat{g})=g_0+\delta \bar{h} +o(\delta)$ which is a Ricci-de Turck flow, then we can find a $h$ satisfying (\ref{ee2}) and a Ricci flow $\hat{g}$ by constructing a family of diffeomorphisms. \\

 For this family of diffeomorphisms $\varphi_\delta: M \rightarrow M$, we define $y(x,\delta)=\varphi_\delta(x)$. The Ricci flow equation implies
\[
\frac{\partial}{\partial t}g_{ij}=-2R_{ij}+\nabla_iV_j+\nabla_jV_i,
\]
where $V_i=\frac{\partial y^\alpha}{\partial t}\frac{\partial x^k}{\partial y^\alpha}g_{ik}$.\\

Following the idea of de Turck, we need to find a $y$ which is governed by
\begin{align}\label{ode}
&\frac{\partial y^\alpha}{\partial t}=\frac{\partial y^\alpha}{\partial x^k}g^{jl}(\Gamma^k_{jl}-\dot{\Gamma}^k_{jl})\\
&\lim_{t\rightarrow -\infty}y^\alpha(x, t)=x^\alpha
\end{align}
where $\dot{\Gamma}^k_{jl}$ are the Christoffel symbols of $g_0$. Since the difference between two connections is a tensor, we have $A^k_{jl}=(\Gamma^k_{jl}-\dot{\Gamma}^k_{jl})$ is a (2,1)-tensor. So the equation (\ref{ode}) is coordinate free. Equation (\ref{ode}) is not a standard initial value problem of first order PDEs, but we can still prove that the solution exists in our case. We will show the existence of solution in section 6.2.\\

Now we suppose that $y$ exists for $t\in(-\infty, N]$, then $V_i=g_{ik}g^{nl}(\Gamma^k_{nl}-\dot{\Gamma}^k_{nl})$. Therefore, we have
\[
\frac{\partial}{\partial \delta}(\frac{\partial}{\partial t}g_{ij})|_{\delta=0}=\frac{\partial}{\partial \delta}(\frac{\partial}{\partial \delta}g_{ij}\lambda\delta)|_{\delta=0}=-\lambda\frac{\partial}{\partial t}g_{ij}|_{\delta=0}
\]
which will equal
\[
\frac{\partial}{\partial \delta}(-2R_{ij}+\nabla_iV_j+\nabla_jV_i)|_{\delta=0}.
\]
Notice that
\[
\frac{\partial}{\partial \delta}V_i|_{\delta=0}=g_{ik}g^{nl}(\frac{\partial}{\partial \delta}\Gamma^k_{nl})|_{\delta=0}
\]
Here we can get
\[
\frac{\partial}{\partial \delta}\Gamma^k_{nl}|_{\delta=0}=\frac{1}{2}g^{kp}(\nabla_n\bar{h}_{pl}+\nabla_l\bar{h}_{np}-\nabla_p\bar{h}_{nl})|_{\delta=0}.
\]
So
\[
g_{ik}g^{nl}(\frac{\partial}{\partial \delta}\Gamma^k_{nl})|_{\delta=0}=\frac{1}{2}g^{nl}(\nabla_n\bar{h}_{il}+\nabla_l\bar{h}_{ni}-\nabla_i\bar{h}_{nl})|_{\delta=0}= \frac{1}{2}[-2(\zeta \bar{h})_i-\nabla_i\bar{H}]
\]
which implies
\[
\nabla_j(\frac{\partial}{\partial\delta}V_i)|_{\delta=0}= -\nabla_j(\zeta \bar{h})_i -\frac{1}{2}\nabla_j\nabla_i \bar{H}.
\]
Similarly, we have
\[
\nabla_i(\frac{\partial}{\partial\delta}V_j)|_{\delta=0}= -\nabla_i(\zeta \bar{h})_j -\frac{1}{2}\nabla_i\nabla_j \bar{H}.
\]
These show that we can use de Turck's method to eliminate $\nabla_{k}\nabla_{l}\bar{H}+\nabla_{k}(\zeta \bar{h})_l+\nabla_{l}(\zeta \bar{h})_k$. Therefore, $\bar{h}$ is a eigenvector of the Linchnerowicz Laplacian
\[
(\Delta_L\bar{h})_{kl}=-\lambda \bar{h}_{kl}.
\]

We will only focus on equation (\ref{mee}) in this paper. For convenience, we replace $\bar{h}$ by $h$ in following paragraphs.\\

\subsection{Existence of modified eigenvector:}
In this section, we prove the existence of $h$ in (\ref{mee}). First of all, since $g_0$ is a Ricci flat metric, we can simplify $(\Delta_Lh)$ as follows
\[
(\Delta_Lh)_{kl}=\Delta h_{kl}+2R^{r\mbox{ }p}_{\mbox{ }k\mbox{ }l}h_{rp}.
\] 
Now, if we define a functional $\mathbf{a}$ by
\begin{align}\label{funl}
\mathbf{a}(h)=\int_M |\nabla h|^2-2R^{\alpha\mu\beta\nu}h_{\alpha\beta}h_{\mu\nu},
\end{align}
equation (\ref{mee}) will be the variational equation for (\ref{funl}). Our goal is finding the minimizer of this functional. Here we set up some notations.\\

\noindent {\bf Notations 1 :} In this article, the integration uses the following simplified notation
\[
\int_M f=\int_M f \sqrt{g_0} dt\wedge dr\wedge d\Omega.
\]
Here $f$ is a function. Also, sometime we use $\dot{g}$ to denote $g_0$ because the lower 0 may be cumbersome when we have some other indices.\\

\begin{definition} We use $C^{\infty}(Sym^2(T^*M))$ and $C^{\infty}_c(Sym^2(T^*M))$ to denote the space of symmetric smooth 2-tensors and the space of compactly supported symmetric smooth 2-tensors. Let $|h|=|\dot{g}^{ij}\dot{g}^{kj}h_{ik}h_{jl}|^{\frac{1}{2}}$ for any $h\in C^{\infty}(Sym^2(T^*M))$. Also, we define $|\nabla h|=|\dot{g}^{ij}\dot{g}^{kl}\dot{g}^{pq}\nabla_ph_{ik}\nabla_qh_{jl}|^{\frac{1}{2}}$. We will also use the following weighted norms
$\|h\|_{\mathbb{H}^1}=(\int_M|\nabla h|^2)^{\frac{1}{2}}+(\int_Mr^{-2}|h|^2)^{\frac{1}{2}}$ and 
$\|h\|_{\mathbb{H}^{0}}=(\int_Mr^{-2}|h|^2)^{\frac{1}{2}}$. 
\end{definition}

Here we choose the function space. First we notice that $M\simeq \mathbb{R}\times \mathbf{S}^2$. We let $\mathbf{SO}(3)$ be the rotation group acting on $\mathbf{S}^2$ part. We can only choose those 2-tensors which are invariant under this action. That means we choose all $h$ such that $\mathbf{g}^*(h)=h$ for all $\mathbf{g}\in \mathbf{SO}(3)$. We call a tensor is radially symmetric if it satisfies this property. Moreover, we have $\mathbf{a}(\mathbf{g}^*(h))=\mathbf{a}(h)$ for all $\mathbf{g}\in \mathbf{SO}(3)$. Now, because our metric $g_0$ is radially symmetric, by [9], we can see that the critical point(minimizer) will be radially symmetric, too. 

\begin{definition}
We define space $\mathbb{H}^1$ as
\[
\mathbb{H}^1=\mbox{closure of }\{ h\in C^{\infty}_c(Sym^2(T^*M)|\mbox{ } h \mbox{ is radially symmetric and } \|h\|_{\mathbb{H}^1}<\infty \}.
\]
 We also write the $L^2$-norm of $h$ as $\|h\|_{2}=(\int_M|h|^2)^{\frac{1}{2}}$.
\end{definition} 

Before we start our proof of the existence, we need some lemmas. We make some remarks here. First of all, we consider $\mathbb{H}^1$-norm under $p$-coordinate. In this case, the $\mathbb{H}^1$-norm controls the integral $\int_M|\nabla h|^2$. Assuming that $h$ is radially symmetric (and recall that $\sqrt{g_0}=O(p)$ near $p=0$), the boundedness of $\int_M|\nabla h|^2$ near $p=0$ implies the boundedness of $\int_0^{\frac{1}{2}}(\frac{\partial}{\partial p}|h|)^2pdp$. Since there is a weighted function $p$ in our integral, we cannot apply Sobolev embedding directly and get $|h|\in C^{0,\frac{1}{2}}$.\\

However, if we have a sequence of radially symmetric 2-tensors 
\[
h^{(z)}\in  C^{\infty}_c(Sym^2(T^*M)),  z\in \mathbb{N}
\]
which are uniformly $\mathbb{H}^1$-bounded, then $(\int_0^{\frac{1}{2}}(\frac{\partial}{\partial p}|h^{(z)}|)^2\sqrt{g_0}dp)^{\frac{1}{2}} + (\int_0^{\frac{1}{2}}|h^{(z)}|^2\sqrt{g_0}dp)^{\frac{1}{2}}$ are uniformly bounded.  So we can find a subsequence of $\{h^{(z)}\}$ and $h\in \mathbb{H}^1$ such that $ (\int_0^{\frac{1}{2}}(|h^{(z)}|-|h|)^2\sqrt{g_0}dp)^{\frac{1}{2}}\rightarrow 0$.  Therefore we have the following lemma. Notice that $r=\frac{4}{3}$ when $p=\frac{1}{2}$

\begin{lemma}
Let $\{h^{(z)}\in C^{\infty}_c(Sym^2(T^*M))$, $z\in \mathbb{N}\}$ be a sequence of radially symmetric 2-tensors satisfying $\|h^{(z)}\|_{\mathbb{H}^1}\leq C$. We also define $M_{\frac{4}{3}}=M\cap\{r\leq \frac{4}{3}\}$. Then we can find  $h\in \mathbb{H}^1$ such that
\begin{align}
\int_{M_{\frac{4}{3}}}(|h^{(z)}|-|h|)^2\rightarrow 0
\end{align}
as $z\rightarrow \infty$.
\end{lemma}

Next, we prove a Hardy type inequality.
\begin{lemma}
Let $h \in C^{\infty}_c(Sym^2(T^*M))$ be a radially symmetric 2-tensor. Then we have
\begin{align}\label{hie}
\int_M  |\nabla h|^2 \geq \int_M \frac{|h|^2}{r^2}
\end{align}
\end{lemma}
\begin{proof}First, we change our coordinate by taking $p^2=(1-r^{-1})$, $p\in (0,1)$. We have
\[
g_0=p^2dt^2+\frac{1}{4}\frac{1}{(1-p^2)^2}dp^2+\frac{1}{(1-p^2)^2}d\Omega^2
\]
So we have $dVol=w_2 \frac{p}{2(1-p^2)^3}dt\wedge dp\wedge d\Omega$ where $w_2$ is the area of $S^2$.\\

Therefore $r^{-1}=(1-p^2)$. We have
\begin{align*}
\int_M \frac{|h|^2}{r^2}=\int_M |h|^2(1-p^2)^2&=\int_M |h|^2 \frac{p}{2(1-p^2)}dtdpd\Omega\\
&=\int_M |h|^2 \partial_p (-\frac{1}{4}\log(1-p^2))dtdpd\Omega\\
&=\int_M\langle  \nabla_1 h, h \rangle \frac{1}{2}\log(1-p^2)dtdpd\Omega\\
&\leq \int_M |\nabla h||h| \frac{1}{4}\frac{|\log(1-p^2)|}{(1-p^2)}dtdpd\Omega.
\end{align*}
Now we notice that
\[
\frac{1}{4}|\log(1-p^2)|\leq \frac{p}{2(1-p^2)}
\]
when $p\in (0,1)$. So we have
\begin{align*}
&\int_M |\nabla h||h| \frac{|\log(1-p^2)|}{4(1-p^2)}dtdpd\Omega\\
\leq &(\int_M |\nabla h|^2)^{\frac{1}{2}}(\int_M | h|^2 (1-p^2)^2)^{\frac{1}{2}}.
\end{align*}
This completes our proof.
\end{proof}

Now, we can start to prove the solution of equation (\ref{mee}) exists by finding a minimizer.

\begin{theorem}\ \\
$\mathbf{a}$. The functional $\mathbf{a}$ defined by (\ref{funl}) has a minimizer on the set\\ 
$\mbox{ }\mbox{ }\mbox{ }$\ \ \ $\mathbb{H}^1\cap \{ \|h\|_{2}=1 \}$.\\
$\mathbf{b}$. We can write the minimizer as\\
$\mbox{ }\mbox{ }$\ \ \ $h=h_0dt^2+h_1dr^2+h_2d\Omega^2$, where $h_0, h_1, h_2$ are function depending on $r$\\
$\mbox{ }\mbox{ }$\ \ \ only.
\end{theorem}
Before we prove this theorem, we prove the following lemma.
\begin{lemma}
There is a $\hat{h} \in \mathbb{H}^1$ such that $\mathbf{a}(\hat{h})<0$.
\end{lemma}
\begin{proof} First of all, we define $\bar{h}\in C^{\infty}(Sym^2(T^*M))$ as follows
\begin{align*}
\bar{h}_{ii}&=\dot{g}_{ii}\mbox{  }\mbox{  }\mbox{   if } i=0,1\\
\mbox{ }&=-\dot{g}_{ii}\mbox{ if }i=2,3
\end{align*}
We can compute the covarient derivatives of $\bar{h}$ and get
\begin{equation}\label{cd}
\begin{array}{ll}
\nabla_2 \bar{h}_{12}&=\nabla_2 \bar{h}_{21}=2r;\\
\nabla_3 \bar{h}_{13}&=\nabla_3 \bar{h}_{13}=2r\sin^2\theta;\\
\nabla_k \bar{h}_{ij}&=0 \ \ \mbox{ otherwise}.
\end{array}
\end{equation}
Now we define $\hat{h}=\eta(r)\bar{h}$, where $\eta$ is a piecewise smooth, radially symmetric function defined as follows
\begin{align*}
\eta(r)&=n(r-1) \mbox{ for } r\in [1,1+\frac{1}{n});\\
&=1 \mbox{ for } r\in [1+\frac{1}{n},\sqrt{2});\\
&=e^{\frac{2\sqrt{2}}{3}}e^{-\frac{2}{3}r} \mbox{ for } r\in [\sqrt{2},\infty)
\end{align*}
where $n>0$ will be determined later. Then we have
\begin{align*}
\mathbf{a}(\hat{h})&=\int_M |\nabla \hat{h}|^2-2\int_M R^{ijij}\hat{h}_{ii}\hat{h}_{jj}\\
&=\int_M (\eta')^2(g_0)^{11}|\bar{h}|^2+\int_M \eta^2|\nabla \bar{h}|^2-2\int_M \eta^2R^{ijij}\bar{h}_{ii}\bar{h}_{jj}\\
&=4\pi w_2[16\int_1^{\infty}\eta^2(1-r^{-1})dr+4\int_1^{\infty} (\eta')^2(1-r^{-1})dr-16\int_1^{\infty}\eta^2r^{-1}dr]\\
&=4\pi w_2[16\int_1^{2}\eta^2 dr+4\int_1^{2} (\eta')^2(1-r^{-1})dr-32\int_1^{2}\eta^2r^{-1}dr]
\end{align*}
where $w_2$ is the area of $\mathbf{S}^2$.\\

 Let $I_1=[1,1+\frac{1}{n})$, $I_2=[1+\frac{1}{n},\sqrt{2})$ and $I_3=[\sqrt{2},\infty)$. We can rewrite integrals on the right hand side of the previous equation as
\begin{align*}
&16\int_1^{2}\eta^2 dr+4\int_1^{2} (\eta')^2(1-r^{-1})dr-32\int_1^{2}\eta^2r^{-1}dr\\
=&\int_{I_1}+\int_{I_2}+\int_{I_3}16\eta^2 dr+\int_{I_1}+\int_{I_3}4(\eta')^2(1-r^{-1})dr-\int_{I_1}-\int_{I_2}-\int_{I_3}32\eta^2r^{-1}dr\\
:=& J_1+J_2+J_3+J_4+J_5+J_6+J_7+J_8.
\end{align*}

We prove that $ J_1+J_2+J_3+J_4+J_5+J_6+J_7+J_8 < 0$. First, because $r^{-1}\leq\frac{1}{2}$ for all $r\in[1,\sqrt{2}]$, we have
\begin{align}\label{ne1}
J_1+J_6\leq 0.
\end{align}

Next, since $\eta(r)=1$ for all $r\in I_2$, we have
\[
J_2=\int_{I_2}16 dr= 16(\sqrt{2}-1-\frac{1}{n}).
\]
Now we consider $J_7$. Again, because $\eta(r)=1$ on $I_2$, we have
\[
J_7=-\int_{I_2}32r^{-1}dr=-32(\log(\sqrt{2})-\log(1+\frac{1}{n}))\leq -\frac{32}{(\sqrt{2})}(\sqrt{2}-1-\frac{1}{n}).
\]
We consider $J_4$. Since $\eta'(r)=n$ for all $r\in I_1$ and $1-r^{-1}\leq r-1$, we have
\[
J_4=4n^2\int_{I_1}(1-r^{-1})dr\leq 4n^2\int_{I_1}(r-1)dr\leq 2.
\]
Therefore we have
\begin{align*}
J_2+J_4+J_7&\leq 16(\sqrt{2}-1-\frac{1}{n})+2-\frac{32}{(\sqrt{2})}(\sqrt{2}-1-\frac{1}{n})\\
&=(16-\frac{32}{(\sqrt{2})})(\sqrt{2}-1-\frac{1}{n})+2
\end{align*}
$(16-\frac{32}{(\sqrt{2})})(\sqrt{2}-1)\approx -2.7451...$. If we take $n$ sufficient large, we will have
\begin{align}\label{ne2}
J_2+J_4+J_7\leq -0.7
\end{align}

Finally, we prove $J_3+J_5+J_8 < 0.7$. Since
\begin{align*}
&J_3+J_5+J_8\\
=&e^{\frac{2\sqrt{2}}{3}}[16\int_{\sqrt{2}}^{\infty} e^{-\frac{4}{3}r}dr+4\int_{\sqrt{2}}^{\infty}(\frac{2}{3})^2e^{-\frac{4}{3}r}(1-r^{-1})dr-32\int_{\sqrt{2}}^{\infty} e^{-\frac{4}{3}r}r^{-1}dr]\\
=&e^{\frac{4\sqrt{2}}{3}}[16(\frac{1}{9}+1)\int_{\sqrt{2}}^{\infty} e^{-\frac{4}{3}r}dr-16(\frac{1}{9}+2)\int_{\sqrt{2}}^{\infty} e^{-\frac{4}{3}r}r^{-1}dr]
\end{align*}
$\int_{\sqrt{2}}^{\infty} e^{-\frac{4}{3}r}r^{-1}dr\approx 0.05734...$ and $e^{\frac{4\sqrt{2}}{3}}\approx 6.5903...$. Therefore we can compute
\begin{align}\label{ne3}
J_3+J_5+J_8\leq 0.5694...
\end{align}

By (\ref{ne1}),(\ref{ne2}) and (\ref{ne3}). we have
\[
J_1+J_2+...+J_8 < -0.1.
\]
This implies $\mathbf{a}(\hat{h})<0$.
\end{proof}

\begin{proof}{(of Theorem 3.5)} Part $\mathbf{a}$: By (\ref{ie1}), we have $\mathbf{a}(h)\geq -2\|h\|_2= -2$ for all $h \in \mathbb{H}^1\cap \{ \|h\|_{2}=1 \}$. So we can find a minimizing sequence $\{h^{(z)}\}_{z \in \mathbb{N}}$:
\[
\lim_{n \rightarrow \infty} \mathbf{a}(h^{(z)})=\min_{h \in \mathbb{H}^1\cap \{ \|h\|_{2}=1 \}} \mathbf{a}(h).
\]
By lemma 3.6, we can assume that $ \mathbf{a}(h^{(z)})<0$ for all $z$. Recall that $|R_{ijij}|\leq r^{-3}\dot{g}_{ii}\dot{g}_{jj}$. We then have
\begin{align*}
0>\mathbf{a}(h^{(z)})&\geq \int_M |\nabla h^{(z)}|^2-r^{-3}\dot{g}^{kk}\dot{g}^{ll}h^{(z)}_{kk}h^{(z)}_{ll}\\
&\geq \int_M|\nabla h^{(z)}|^2 -\int_M r^{-3}(|\dot{g}^{kk}\dot{g}^{kk}h^{(z)}_{kk}h^{(z)}_{kk}|^2+|\dot{g}^{ll}\dot{g}^{ll}h^{(z)}_{ll}h^{(z)}_{ll}|^2)\\
&\geq \int_M|\nabla h^{(z)}|^2 -2\int_M r^{-3}|h^{(z)}|^2 
\end{align*}
That is
\begin{align}\label{wiq}
\|\nabla h^{(z)}\|_2 < 2 \|h^{(z)}\|_{\mathbb{H}^0},
\end{align}
Here $\|h^{(z)}\|_{\mathbb{H}^0}\leq \|h^{(z)}\|_{2}= 1$ for all $z$. By (\ref{wiq}), we get $\|h^{(z)}\|_{\mathbb{H}^1}<3$ so $\|h^{(z)}\|_{\mathbb{H}^1}$ is bounded. Therefore there is a weak limit
\[
h^{(z)}\rightharpoonup h
\]
in $\mathbb{H}^1$ satisfying $\liminf_{z \rightarrow \infty}\|h^{(z)}\|_{\mathbb{H}^1}\geq\|h\|_{\mathbb{H}^1}$. Also, for any compact set $\Omega \subset M $, the boundedness of  $\|h^{(z)}\|_{\mathbb{H}^1;\Omega}$ implies that there is a strongly convergent subsequence in $L^2(\Omega)$. Combining with lemma 3.3, we have
\[
\int_{M_R} |h^{(z)}|^2 \rightarrow \int_{M_R} |h|^2
\]
and
\begin{align}
\int_{M_R} R^{klkl}h^{(z)}_{kk}h^{(z)}_{ll} \rightarrow \int_{M_R} R^{klkl}h_{kk}h_{ll}
\end{align}
for all $R>1$ as $z$ tends to infinity. Also, we can easily get $\|h\|_{2}\leq 1$.\\

Now fix $\varepsilon>0$, using the fact $\liminf_{z \rightarrow \infty}\|h^{(z)}\|_{\mathbb{H}^1}\geq\|h\|_{\mathbb{H}^1}$, we have
\[
\int_M|\nabla h^{(z)}|^2 +\int_Mr^{-2}| h^{(z)}|^2 +\varepsilon \geq \int_M|\nabla h|^2+\int_Mr^{-2}| h|^2
\]
for all sufficient large $z$. Since $\|h\|_2\leq \|h^{(z)}\|_2=1$, we can choose $M_R= M\cap\{r\leq R\}$ with $R$ sufficient large such that $\int_{M-M_R}r^{-2}| h^{(z)}|^2 \leq R^{-2}\|h^{(z)}\|^2_{2}\leq \varepsilon$ and $\int_{M-M_R}r^{-2}| h|^2 \leq R^{-2}\|h\|^2_{2}\leq \varepsilon$ for all $i$. So
\[
\int_M|\nabla h^{(z)}|^2 +\int_{M_R} r^{-2}| h^{(z)}|^2 +3\varepsilon \geq \int_M|\nabla h|^2+\int_{M_R}r^{-2}| h|^2
\]
This implies
\[
\liminf_{z\rightarrow \infty}\int_M|\nabla h^{(z)}|^2 +3\varepsilon \geq \int_M|\nabla h|^2.
\]
Because $\varepsilon$ is arbitrary, $\liminf_{z\rightarrow \infty}\int_M|\nabla h^{(z)}|^2 \geq \int_M|\nabla h|^2$. By (3.14), we get $\mathbf{a}(h)\leq \lim_{z\rightarrow \infty}\mathbf{a}(h^{(z)})$.\\

Next, we will show that $h$ is not zero. Here we use the Hardy type inequality (\ref{hie}). We can assume that there exists $\varepsilon>0$ such that
\[
-\varepsilon > \mathbf{a}(h^{(z)})\geq \int_M |\nabla h^{(z)}|^2 -2\int_M r^{-3}|h^{(z)}|^2
\]
for all  sufficient large $z$. By (\ref{hie}),
\[
 \int_M |\nabla h^{(z)}|^2 -2\int_M r^{-3}|h^{(z)}|^2\geq \int_M r^{-2}|h^{(z)}|^2-2\int_M r^{-3}|h^{(z)}|^2
\]
This implies that
\[
\int_M (2r^{-3}-r^{-2})|h^{(z)}|^2 \geq \varepsilon
\]
Consider the positive part of left hand side, i.e. $\{r\leq 2\}\cap M=S$. We have
\[
3\int_S |h^{(z)}|^2 \geq \varepsilon
\] 
for large $z$. Taking limit we get $\|h\|_2 > \frac{\varepsilon}{3}$.
Therefore if we set $\bar{h}=\frac{h}{\|h\|_2}$, we will have $\mathbf{a}(\bar{h})\leq \mathbf{a}(h)\leq \mathbf{a}(h^{(z)})$. To make our notation simple, we replace the notation $\bar{h}$ by $h$, which is the minimizer we found.\\

Part $\mathbf{b}$: First of all, notice that we can write $h= h_{ij}dx^i\otimes dx^{j}$, with $x^0=t$, $x^1=r$, $x^2=\theta$, $x^3=\varphi$. According to our computations of the curvature in section 2, the negative term of $\mathbf{a}(h)$ is contributed only by the diagonal part of $h$ (We have only one coordinate chart, so $diag(h)$ makes no confusion). Therefore we have
\[
0>\mathbf{a}(h) \geq \mathbf{a}(diag(h));
\]
\[
\|h\|_2 \geq \|diag(h)\|_2.
\]
So $h$ is a minimizer iff $h=diag(h)$. Next, we prove that $h_{ii}$, $\forall i$, are independent of $t$. This is also easy to see. We prove this by contradiction. If not, we can choose $t_0\in [0,4\pi]$ such that it minimizes
\[
S(\tau)=\int_{M\cap \{t=\tau\}} |\nabla h|^2-2R^{\alpha\mu\beta\nu}h_{\alpha\beta}h_{\mu\nu} rdrd\Omega. 
\]
We define $h^*=h(t_0)$. Because there $h$ doesn't change along the $t$ direction, we have $|\nabla h^*|\leq |\nabla h|$. So
\[
\mathbf{a}(h^*)\leq \mathbf{a}(h).
\] 
"$=$" iff $h_{ii}$ are independent of $t$.\\
Combining these two facts and $h$ is radially symmetric, we prove part $\mathbf{b}$.
\end{proof}

\section{Short-Time Existence}
\subsection{Estimate of $|h|$ and $|\nabla h|$}
In the previous section, we prove that there is $h\in \mathbb{H}^1$ such that $\Delta_Lh=-\lambda h$ and $h$ is radially symmetric. Using the Sobolev inequalities, we will prove that $h\in C^{0,\frac{1}{2}}(\Omega)$ for any compact subset $\Omega\subset M$. Using ideas introduced by Schauder, we will then prove further regularity property of $h$ locally (Since $\Delta_Lh=-\lambda h$ is a elliptic PDE and $h\in C^{0,\frac{1}{2}}(\Omega)$, we have $h\in C^{2,\frac{1}{2}}(\Omega)$. Repeating this process, we will prove that $h\in C^{\infty}(\Omega)$). Notice that, even though $h$ is obtained by taking a weak limit of a sequence of compactly supported 2-tensors, it doesn't indicate that $|h|$ will vanish as $r\rightarrow 1$ and $r\rightarrow \infty$. We will show in this section that such is indeed the case. \\

Here we start with estimates of $h$. We generalize our arguments a little bit such that we can apply these arguments in next section.

\begin{definition} Let $\|k\|_{W^{1,2}}=\|k\|_2+\|\nabla k\|_2$ and $\|k\|_{W^{2,2}}=\|k\|_2+\|\nabla k\|_2+\|\nabla\nabla k\|_2$ for all $k \in C^{\infty}_0(Sym^2(T^*M))$. Here we define
\[ 
|\nabla\nabla k|^2=\dot{g}^{\alpha\beta}\dot{g}^{\gamma\delta}\dot{g}^{\lambda\mu}\dot{g}^{\nu\xi}\nabla_\alpha\nabla_\gamma k_{\lambda\nu} \nabla_\beta\nabla_\delta k_{\mu\xi}
\]
and $\|\nabla\nabla k\|_2=(\int_M|\nabla\nabla k|^2)^{\frac{1}{2}}$.
\end{definition}

\begin{pro} Let $k \in C^{\infty}_0(Sym^2(T^*M))$ be a radially symmetric 2-tensor. Suppose $|k|$, $|\nabla k|$ and $|\nabla \nabla k|$ are $L^2$ functions. Then
\begin{align}
|k|\leq C \| k\|_{W^{1,2}},\\
|\nabla k|\leq C \| k\|_{W^{2,2}}
\end{align}
 for some universal constant $C$ (When we use the term "universal constant", it means a constant that depends only on $g_0$).
\end{pro}
\begin{proof} We use $p$-coordinate as we did in lemma 3.5 and use the divergence theorem. We define $Y=|k|(\dot{g}^{11})^{\frac{1}{2}}\partial_1$. Then if we set $M_x=M\cap \{p < x\}$, $x\leq \frac{1}{2}$ we have 
\begin{align}\label{div}
\int_{M_x} div(Y) =\int_{\partial M_x} i_YVol,
\end{align}
where
\begin{align*}
div(Y)&= \nabla_1(Y^1)= (\nabla_1 |k|)(\dot{g}^{11})^{\frac{1}{2}}+ |k|^2 \nabla_1((\dot{g}^{11})^{\frac{1}{2}}\partial_1)^1\\
&\leq \frac{1}{2}|k|^{-\frac{1}{2}}\langle \nabla_1k,k\rangle (\dot{g}^{11})^{\frac{1}{2}} \leq \frac{1}{2}|\nabla k|
\end{align*}
So the left hand side of (\ref{div}) is smaller than
\[
\frac{1}{2}\int_{M_x}|\nabla k| \leq \frac{1}{2}(\int_{M_x}|\nabla k|^2)^{\frac{1}{2}}(\int_{M_x}1)^{\frac{1}{2}}\leq C \|\nabla k\|_{2}x.
\]
Now consider the right hand side, we have
\begin{align*}
|\int_{\partial M_x} i_YVol|&= |\int_{\{p=0\}}|k|\frac{p}{(1-p^2)^2}dtd\Omega-\int_{\{p=x\}}|k|\frac{p}{(1-p^2)^2}dtd\Omega|\\&= 4\pi\omega_2\frac{x}{(1-x^2)^2}|k|(x)\geq C 4\pi\omega_2 x|k|(x)
\end{align*}
where $w_2$ is the area of $S^2$. So (4.1) is valid when $p\leq \frac{1}{2}$.\\

When $p\geq \frac{1}{2}$, we consider $Y=|k|^2\dot{g}^{11}\partial_1$ and $M^x=M\cap\{p>x\}$. Use the similar argument, we have
\[
\int_{M^x} div (Y) \leq \int_{M^x}|\nabla k|^2+|k|^2.
\]

Here we claim that $\liminf_{p\rightarrow 1}\frac{|k|^2}{(1-p^2)^2}=0$. If not, we will have
$|k|^2\geq c(1-p^2)^2$ for some positive $c$. That means $\int_M |k|^2 = \int_M (1-p^2)^2=\infty$, which is a contradiction.\\

Therefore we have
\[
\int_{\partial M^x} i_YVol= -C\frac{|k|^2}{(1-x^2)^2}.
\]
So we get (4.1).
\end{proof}

To prove the second inequality, We use the similar argument as about by replacing $k$ by $\nabla k$.\\

\begin{remark} Using the $p$-coordinate, since $\dot{g}_{00}=p^2$, we have $h''_{00}$ is bounded near $p=0$. (That is, near $r=1$.)
\end{remark}

We can apply this proposition to our $h$.
\begin{cor}
Let $h$ be the eigenvector provided by theorem 3.5. Then there exist $C_1$, $C_2(\lambda),C_3(\lambda)>0$ such that
\begin{align}
|h|\leq C_1,\\
|\nabla h|\leq C_2,\\
\ \ |\nabla\nabla h|,\ |\nabla^{(3)} h|\leq C_3.
\end{align}
where $C_2$, $C_3$ depends on $\lambda$ and $C_1$ is universal.
\end{cor}
\begin{proof} We apply proposition 4.2. Notice that $\|h\|_{2}=1$ and $\|\nabla h\|_2\leq 1$, so we get the first inequality. Moreover, since equation (\ref{mee}) tell us that $|\nabla\nabla h|$ can be controlled by $|h|$ (Here we use the fact that $h$ is radially symmetric), so $\|\nabla\nabla h\|_2\leq C_2(\lambda)$. We can get the second inequality.\\

Since $h$ is radially symmetric, $|\nabla\nabla h|$ can be controlled by $(\Delta_Lh)$. Similarly, $|\nabla^{(3)} h|$ can be controlled by $\nabla(\Delta_Lh)$. Now apply (4.4) and (4.5) in equation $(\Delta_Lh)=-\lambda h$ and $\nabla(\Delta_Lh)=-\lambda\nabla h$. We get (4.6). 
\end{proof}

\subsection{Vanishing of $h$ at infinity}
We can improve our estimate near $p=0$ and $p=1$. First of all, because $h$ is radially symmetric, the equation $\nabla^{(l)}(\Delta_Lh)=-\lambda \nabla^{(l)}h$ tells us that $|\nabla^{(l+2)}h|$ is $L^2$ bounded if $|\nabla^{(l)}h|$ is $L^2$ bounded. So by induction, we have $|\nabla^{(l)}h|$ is $L^2$ bounded for all $l$ which implies $\|k\|_{W^{l+1,2}}\leq C_{l,\lambda}$ for some constant $C_{l,\lambda}$ depending on $l$ and $\lambda$. Notice that, if we replace $Y$ by $|\nabla^{(l)}h|(\dot{g}^{11})^{\frac{1}{2}}\partial_1$ in equation (\ref{div}) and follow the argument of proposition 4.2, we will get
\[
|\nabla^{(l)}k|\leq C \|k\|_{W^{l+1,2}}.
\]
Now, because $\|k\|_{W^{l+1,2}}\leq C_{l,\lambda}$, corollary 4.3 can be generalized such that
\[
|\nabla^{(l)}h|\leq C_{l,\lambda}
\]
for some constant $C_{l,\lambda}$ depending on $l$ and $\lambda$.
\begin{pro} Let $h$ be the eigenvector provided by theorem 3.5, we have
\begin{align}
|\nabla^{(l)}h|(p)&\leq C_{l,\lambda} p.
\end{align}
for $p\leq \frac{1}{2}$ and a constant $C_{l,\lambda}$ depending only on $l$ and $\lambda$. Therefore, $|\nabla^{(l)}h|\rightarrow 0$ as $p\rightarrow 0$.
\end{pro}
\begin{proof}We use equation (\ref{div}) in proposition 4.2 again. Since
\[
|\nabla h|(p)\leq C\|h\|_{W^{2,2}},
\]
we have
\[
\frac{1}{2}\int_{M_x}|\nabla h| \leq \frac{1}{2}\|h\|_{W^{2,2}}\int_{M_x}1\leq C\|h\|_{W^{2,2}} p^2
\]
by taking $x=(1-p^2)^{-1}$ for $p<\frac{1}{2}$. This gives us the estimate on the left hand side of (\ref{div}).\\

Since we already know that the right hand side of (4.3) can be expressed as $p(\frac{1}{1-p^2})|h|(p)$, we have
\[
|h|(p)\leq C\|h\|_{W^{2,2}} p. 
\]

To get the estimate of $|\nabla^{(l)}h|$ for $l>1$, we replace $Y$ by $|\nabla^{(l)}h|(\dot{g}^{11})^{\frac{1}{2}}\partial_1$ in (\ref{div}) and follow the the same argument. Then we can get this result.
\end{proof}

We finish this subsection by giving a similar estimate of $|h|$ near $p=1$.
\begin{pro}
Let $h$ be the eigenvector provided by theorem 3.5, we have
\[
|\nabla^{(l)}h|(p)\leq C_{l,\lambda}(1-p^2)
\]
for $p\geq \frac{1}{2}$ and a constant $C_{l,\lambda}$ depending only on $l$ and $\lambda$. Therefore we have $|\nabla^{(l)}h|\rightarrow 0$ as $p\rightarrow 1$.
\end{pro}
\begin{proof}
We can prove this estimate by using the argument of proposition 4.2. Recall that: when $x \geq \frac{1}{2}$, we use divergence theorem by taking $Y=|k|^2\dot{g}^{11}\partial_1$ and $M^x=M\cap\{p>x\}$. We will have
\[
\frac{|k|^2}{(1-x^2)^2}\leq |\int_{M^x} div (Y)| \leq \int_{M^x}|\nabla k|^2+|k|^2.
\]
So we have $|k|(p)\leq C\|k\|_{W^{1,2}}(1-p^2)$.\\

For $l>1$, we replace $Y$ by $|\nabla^{(l)}k|^2\dot{g}^{11}\partial_1$ and follow the the same argument.
\end{proof}

\subsection{Short-time existence}
Now we can prove the short-time existence theorem. Our proof is based on work of Shi [1]. Here we quote two theorems given by Shi.
\begin{theorem}
Let $(M^n,\bar{g})$ be a complete noncompact manifold with $|Rm(\bar{g})|\leq k_0<\infty$, then there exists a constant $T(n,k_0)>0$ such that the Ricci flow equation
\begin{align*}
&\frac{\partial}{\partial t}g_{ij}(x,t)=-2R_{ij}(x,t) \mbox{  }x\in M;\\
&g_{ij}(x,0)=\bar{g}(x) \mbox{ }x\in M
\end{align*}
has a smooth solution on $0\leq t\leq T(n,k_0)$.
\end{theorem}
We have a similar result for Ricci-de Turck flows which is also given by Shi [1].
\begin{theorem}
Let $(M^n,\bar{g})$ be a complete noncompact manifold with $|Rm(\bar{g})|\leq k_0<\infty$, then there exists a constant $T(n,k_0)>0$ such that the Ricci-de Turck flow equation
\begin{align*}
&\frac{\partial}{\partial t}g_{ij}(x,t)=-2R_{ij}(x,t)+\nabla_iV_j+\nabla_jV_i\mbox{  }x\in M;\\
&g_{ij}(x,0)=\bar{g}(x) \mbox{ }x\in M
\end{align*}
where $V_i=g_{ik}g^{pl}(\Gamma^k_{pl}-\bar{\Gamma}^k_{pl})$, has a smooth solution on $0\leq t\leq T(n,k_0)$.
\end{theorem}
Using this result, we get
\begin{theorem}
Let $(M,g_0)$ be the Euclidean Schwarzschild metric, $h$ be the eigenvector (2-tensor) obtained from theorem 3.5 and $\varepsilon>0$. The Ricci-de Turck equation
\begin{align*}
\frac{\partial}{\partial t}g_{ij}&=-2R_{ij}+\nabla_i V_j^{(\varepsilon)}+\nabla_j V_i^{(\varepsilon)}\\
g(x,0)&=(g_0+\varepsilon h)
\end{align*}
where $V_i^{(\varepsilon)}=g_{ik}g^{pl}(\Gamma^k_{pl}-\Gamma_{pl}^{k(\varepsilon)})$, has a solution for short-time when $\varepsilon$ is small enough.
\end{theorem}
\begin{proof} Since we have the boundedness of $|\nabla^{(m)}h|$ for $m=0,1,2,3$ by corollary 4.4, we can prove this theorem by applying Shi's result directly.
\end{proof}

\begin{remark}
We can prove that $g$ is diagonal and radially symmetric by following Shi's argument carefully. In [1], Shi proves the existence of solutions by solving the Dirichlet problems on each $D_i\subset M$, $i\in \mathbb{N}$, where $D_i$ are compact subsets and $\cup D_i = M$. We use $p$-coordinate in our case, we choose a sequences such that $ s_i\uparrow 1$ as $i\rightarrow \infty$. Now if we define $D_i=M\cap \{p \leq s_i\}$, $g$ will be a limit of a sequence of 2-tensors $g_i$ with each $g_i$ satisfies Ricci-de Turck equation on $D_i$ and $|g_i -(g_0+\varepsilon h)|$ vanished on the boundary of $D_i$. Since the initial metric is diagonal and radially symmetric and $D_i$ are radially symmetric, we have $g_i$ are diagonal and radially symmetric. This implies $g$ is diagonal and radially symmetric.
\end{remark}
\begin{remark}
By choosing $D_{i}$ as above, since $g$ is the limit of a sequence $\{g_i\}$ where $|g_i -(g_0+\varepsilon h)|$ vanished on the boundary of $D_i$, we can prove that our solution $g$ actually satisfies $|g-(g_0+\varepsilon h)|\rightarrow 0$ as $r\rightarrow 1$. In the next section, we will prove that  $|g-(g_0+\varepsilon h)|\rightarrow 0$ as $r\rightarrow \infty$ by using the comparison theorem.
\end{remark}
\begin{remark}
 If we compute the curvature $R_{ijij}(g_0+\varepsilon h)$ directly, we will have $|R_{ijij}(g_0+\varepsilon h)-R_{ijij}(g_0)|\leq C \varepsilon$.\\  
\end{remark}
We have the following proposition:
\begin{pro}
Let $k \in C^{\infty}_0(Sym^2(T^*M))$ be a radially symmetric 2-tensor which satisfies
\begin{align}\label{cod1}
|\nabla^{(m)}  k|  \leq \varepsilon\mbox{ for }m=0,1,2,3;
\end{align}
for some $\varepsilon>0$. Then there is universal constants $C>0$ and $N>0$ such that
\[
|R_{ijij}(g_0+k)-R_{ijij}(g_0)|\leq C\varepsilon
\]
provided $\varepsilon<N$.
\end{pro}

\section{Long-Time Existence:}
To prove the long-time existence, first we notice that if the initial curvature $|Rm(\bar{g})|\leq k_0$, then there is a Ricci-de Turck flow $g$ defined on $M\times [0,T]$ where $T$ depends only on $k_0$. Now we check whether $|Rm(g(T))|\leq k_0$ or not. If $|Rm(T)|\leq k_0$, we will prove that $g$ can be extended to $2T$. Then we check the validity of $|Rm(g(2T))|\leq k_0$. If $|Rm(g(2T))|\leq k_0$, we will prove that $g$ can be extended to $3T$. Continue this process, we can prove that the time interval can be extended to be large as desired if we have some control of $|Rm(g)|(t)$.\\

 We start with several definitions.

\subsection{Estimate cones:} In this subsection, we define a special set called the estimate cone. We can see in the following sections that the Ricci-de Turck flow with initial value in an estimate cone will stay in it. Moreover, the structure of this cone will help us to find a convergent subsequence of Ricci-de Turck flows in the next section.\\

\begin{definition}
Let $E$ be a Banach space and $h\in E$. A cone along $h$ with opening $M$ is the set 
\[
C_{h,M}=\{\varphi\in E | \inf_{\delta \in \mathbb{R}^+\cup \{0\}}\frac{\|\varphi-\delta h\|}{\delta}\leq M \}
\]
We also denote the translation of this cone $C_{h,M}(f)= C_{h,M}+f$ and call $h$ the axis of this cone.
\end{definition}
Now we define $E$. Let $k\in C^{\infty}_0(Sym^2T^*M)$ such that $k$ is radially symmetric and diagonal. We consider the estimate cone under the following norm:
\[
\|k\|_{W^{1,2}}=\|k\|_2+\|\nabla k\|_2
\]
Then we define our function space as
\begin{definition}
We define $\mathbb{H}=\mbox{closure of }\{k\in C^{\infty}_0(Sym^2(T^*M))| \|k\|_{W^{1,2}} <\infty \mbox{ and }k \mbox{ is symmetric and diagonal}\}$
\end{definition}
Also, we change our time variable.
\begin{definition}
We define $\delta(t)$ to be $e^{-\lambda t}$ for all $t\in \mathbb{R}$.
\end{definition}

Now we set our axis $h$ to be the eigenvector which is given by theorem 3.5 and change our time variable by $\delta=e^{-\lambda t}$.
 We are interested in those $\varphi \in C_{h,M}(g_0)$.\\

\subsection{Shi's estimates}
Based on theorem 2.5 in Shi's paper and his argument, we have the following theorems.
\begin{theorem} Suppose the conditions of theorem 4.8 are fulfilled. Then for any $\rho>0$, there is a $T(\rho,n,k_0)>0$ such that
\[
|g(x,t)-\bar{g}(x)| \leq \rho
\]
for $x\in M$, $0\leq t\leq T(\rho,n,k_0)$.
\end{theorem}
Shi also gave a proof of the derivative estimates. See [1].
\begin{theorem}
Let $g$ be the solution of equation in theorem 4.9. and $|Rm(\bar{g})|\leq k_0$. Then there is a $T(n,k_0)>0$ depends only on $n$ and $k_0$ such that
\[
\sup_{(x,t)\in M\times [0,T(n,k_0)]}|\bar{\nabla} g| \leq c(n,k_0),
\]
and
\[
\sup_{(x,t)\in M\times[0,T(n,k_0)]}|\frac{\partial}{\partial t}g| \leq c(n,k_0),
\]
where $c(n,k_0)$ depends only on $n$ and $k_0$.
\end{theorem}

\subsection{Evolution equations}
Let $C_1$ be the constant given in corollary 4.4 and $\rho=\min\{\frac{1}{4},\frac{1}{8C_1}\}$. We fix a $\varepsilon <\frac{\rho}{-\lambda}$ where $\lambda$ is the eigenvalue given by theorem 3.5. Let $g$ be the solution of the Ricci-de Turck equation which is given by theorem 4.9 and then shift our time by replacing $t$ by $t_0+t$. So we have $g$ is a Ricci-de Turck flow defined on $M\times [t_0,t_0+T]$.\\

We start with the Ricci-de Turck equation. First of all, the Ricci-de Turck equation is given by
\begin{align*}
\frac{\partial}{\partial t}g_{ij}&=-2R_{ij}+\nabla_iV_j+\nabla_jV_i\\
&= g^{\alpha\beta}\bar{\nabla}_\alpha\bar{\nabla}_\beta g_{ij}-g^{\alpha\beta}g_{ip}\bar{g}^{pq}\bar{R}_{j\alpha q\beta}-g^{\alpha\beta}g_{jp}\bar{g}^{pq}\bar{R}_{i\alpha q\beta}\\
&\mbox{ }\mbox{ }+\frac{1}{2}g^{\alpha\beta}g^{pq}(\bar{\nabla}_ig_{p\alpha}\bar{\nabla}_jg_{q\beta}+2\bar{\nabla}_\alpha g_{jp}\bar{\nabla}_qg_{i\beta}-2\bar{\nabla}_\alpha g_{jp}\bar{\nabla}_\beta g_{iq}\\
&\mbox{ }\mbox{ } \mbox{ }\mbox{ }\mbox{ }\mbox{ }\mbox{ }\mbox{ }\mbox{ }\mbox{ }\mbox{ }\mbox{ } \mbox{ }\mbox{ }\mbox{ }\mbox{ }\mbox{ }\mbox{ }\mbox{ }\mbox{ }\mbox{ }\mbox{ } \mbox{ }\mbox{ }\mbox{ }\mbox{ }\mbox{ }\mbox{ }\mbox{ }\mbox{ } \mbox{ } \mbox{ }\mbox{ }-2\bar{\nabla}_jg_{p\alpha}\bar{\nabla}_\beta g_{iq}-2\bar{\nabla}_ig_{p\alpha}\bar{\nabla}_\beta g_{jp})
\end{align*}
with initial data $\bar{g}=g_0+\varepsilon h$. Here we use $\nabla$, $\bar{\nabla}$ and $\dot{\nabla}$ to denote the covariant derivatives with respect to $g_{ij}$, $\bar{g}_{ij}$ and $(g_0)_{ij}$. Since the metric $g_{ij}$ is diagonal, we have
\begin{align*}
\frac{1}{2}&g^{\alpha\beta}g^{pq}(\bar{\nabla}_ig_{p\alpha}\bar{\nabla}_jg_{q\beta}+2\bar{\nabla}_\alpha g_{jp}\bar{\nabla}_qg_{i\beta}-2\bar{\nabla}_\alpha g_{jp}\bar{\nabla}_\beta g_{iq}\\
&\mbox{ }\mbox{ } \mbox{ }\mbox{ }\mbox{ }\mbox{ }\mbox{ }\mbox{ }\mbox{ }\mbox{ }\mbox{ }\mbox{ } \mbox{ }\mbox{ }\mbox{ }\mbox{ }\mbox{ }\mbox{ }\mbox{ }\mbox{ }\mbox{ }\mbox{ } \mbox{ }\mbox{ }\mbox{ }\mbox{ }\mbox{ }\mbox{ }\mbox{ }\mbox{ } \mbox{ }\mbox{ }-2\bar{\nabla}_jg_{p\alpha}\bar{\nabla}_\beta g_{iq}-2\bar{\nabla}_ig_{p\alpha}\bar{\nabla}_\beta g_{jp})
\end{align*}
is equal to 0 if $i=j\neq 1$ and is equal to $\frac{1}{2}g^{kk}g^{kk}(\bar{\nabla}_1g_{kk}\bar{\nabla}_1g_{kk})$ when $i=j=1$.\\

Now we set $v(x,t)=g(x,t)-\bar{g}(x)$ and $w(x,t)=v(x,t)-(\delta(t)-\varepsilon)h(x) $, where $h$ is the eigenvector provided in theorem 3.5. We want to derive a evolution equation for $w$.\\

First, we consider the equation for $v$. Since $\bar{\nabla}\bar{g}\equiv0$, we have
\begin{align}
\frac{\partial}{\partial t} v_{ii}=g^{\alpha\alpha}\bar{\nabla}_\alpha\bar{\nabla}_\alpha v_{ii}+ A_{ii}
\end{align}
where $A_{ii}$ can be writen as
\[
A_{ii}= -2 g^{\alpha\alpha}g_{ii}\bar{g}^{ii}\bar{R}_{i\alpha i\alpha}
\]
if $i\neq 1$ and
\[
A_{11}=-2 g^{\alpha\alpha}g_{11}\bar{g}^{11}\bar{R}_{1\alpha 1\alpha}+\frac{1}{2}g^{kk}g^{kk}(\bar{\nabla}_1v_{kk}\bar{\nabla}_1v_{kk})
\]
if $i=1$.\\

For $(\delta(t)-\varepsilon)h$, we have
\begin{align*}
\frac{\partial}{\partial t}(\delta(t)-\varepsilon))h_{kk}&=-\lambda \delta(t)h_{kk}=(\delta(t)-\varepsilon)(\Delta_Lh)_{kk}-\lambda\varepsilon h_{kk}\\
&=\dot{g}^{\alpha\beta}\dot{\nabla}_\alpha\dot{\nabla}_\beta (\delta(t)-\varepsilon)h+2(\delta(t)-\varepsilon)\dot{R}^{i\mbox{ }i\mbox{ }}_{\mbox{ }k\mbox{ }k}h_{ii}-\lambda\varepsilon h_{kk}.
\end{align*}
Therefore, we get

\begin{align}
\frac{\partial}{\partial t}w_{kk}=g^{\alpha\alpha}\bar{\nabla}_\alpha\bar{\nabla}_\alpha w_{kk}&+(g^{\alpha\alpha}\bar{\nabla}_\alpha\bar{\nabla}_\alpha-\dot{g}^{\alpha\alpha}\dot{\nabla}_\alpha\dot{\nabla}_\alpha)(\delta(t)-\varepsilon) h_{kk}\\ 
&+(A_{kk}-2(\delta(t)-\varepsilon)\dot{R}^{i\mbox{ }i\mbox{ }}_{\mbox{ }k\mbox{ }k}h_{ii})+\lambda \varepsilon h_{kk}\nonumber.
\end{align}

\subsection{Vanishing of $|g-\dot{g}|$ at infinity} In this subsection, we prove that $|v|=|g-\bar{g}| \rightarrow 0$ as $r\rightarrow \infty$. We start with equation (5.1).
\[
\frac{\partial}{\partial t} v_{ii}=g^{\alpha\alpha}\bar{\nabla}_\alpha\bar{\nabla}_\alpha v_{ii}+ A_{ii}.\\
\]

Using the facts that $|Rm(\dot{g})|\leq Cr^{-3}$ and proposition 4.6, we have $|Rm(\bar{g})|\rightarrow 0$ as $r \rightarrow \infty$. Now, by Theorem 5.4, we have $|v|$, $|\bar{\nabla}v|$ are bounded. Therefore, $v$ is a solution of the following parabolic equation
\begin{align*}
P(v)=\frac{\partial}{\partial t}v_{ii}-g^{\alpha\alpha}\bar{\nabla}_\alpha\bar{\nabla}_\alpha v_{ii}&- \frac{1}{2}\delta_{i1}g^{kk}g^{kk}\bar{\nabla}_1v_{kk}\bar{\nabla}_1v_{kk}\\
&+2 g^{\alpha\alpha}g_{ii}\bar{g}^{ii}\bar{R}_{i\alpha i\alpha}=0
\end{align*}
We denote
\[
L'(v)=g^{\alpha\alpha}\bar{\nabla}_\alpha\bar{\nabla}_\alpha v_{ii}+ \frac{1}{2}\delta_{i1}g^{kk}g^{kk}\bar{\nabla}_1v_{kk}\bar{\nabla}_1v_{kk}
\]
and $2 g^{\alpha\alpha}g_{ii}\bar{g}^{ii}\bar{R}_{i\alpha i\alpha}=F$.\\

So we can compute directly and get the evolution equation for $|v|^2= \bar{g}^{ij}\bar{g}^{kl}v_{ik}v_{jl}$:
\begin{align*}
\frac{\partial}{\partial t}|v|^2&= 2\langle \frac{\partial}{\partial t}v,v\rangle =2\langle L'(v), v\rangle\\
&=2 \langle g^{\alpha\alpha}\bar{\nabla}_\alpha\bar{\nabla}_\alpha v,v\rangle + (\bar{g}^{11})^2 g^{kk}g^{kk}\bar{\nabla}_1v_{kk}\bar{\nabla}_1v_{kk}v_{11}+ \langle F,v\rangle
\end{align*}
We notice that
\[
2\langle g^{\alpha\alpha}\bar{\nabla}_\alpha\bar{\nabla}_\alpha v,v\rangle= g^{\alpha\alpha}\bar{\nabla}_\alpha\bar{\nabla}_\alpha |v|^2- 2g^{\alpha\alpha}\langle \bar{\nabla}_{\alpha}v,\bar{\nabla}_{\alpha}v\rangle.
\]
So by theorem 5.3, we can choose $T_0>0$ depending only on $|Rm(\dot{g})|$ such that 
\[
(\bar{g}^{11})^2 g^{kk}g^{kk}\bar{\nabla}_1v_{kk}\bar{\nabla}_1v_{kk}v_{11}-
2g^{\alpha\alpha}\langle \bar{\nabla}_{\alpha}v,\bar{\nabla}_{\alpha}v\rangle\leq 0
\]
for all $t\in [t_0,t_0+T_0]$. Therefore we have
\begin{align*}
\frac{\partial}{\partial t}|v|^2\leq L(|v|^2)+ \langle F, v\rangle
\end{align*}
for $L=g^{\alpha\alpha}\partial_\alpha\partial_\alpha$.\\

Because $v$ is radially symmetric (by remark 4.10),  $|v|$ depends only on $r$. We can change the variables by taking
\begin{align}
s=\begin{cases}&(1-r^{-1})^{\frac{1}{2}}\mbox{ for } r\leq 2\\
                              &r \mbox{ for } r\geq 3\\
                              &\mbox{a smooth increasing function for } r\in(2,3).
                              \end{cases}
\end{align}
 Then $L$ is uniformly parabolic parabolic by using this coordinate and $|v|^2$ is a function defined on $(s,t)\in (0,\infty)\times [t_0,t_0+T_0]$.\\

Since we want to the use comparison theorem, we should quote the following maximum principle first. The proof of this lemma follows the idea of theorem 6 in section 2.3 of [6].
\begin{lemma}
Suppose $u\in C^{2}(U\times[0,T])$ with an unbounded $U\subset \mathbb{R}^n$ is a solution of the following parabolic equation
\begin{align*}
\frac{\partial}{\partial t} u &\leq a^{ij}(x,t) \frac{\partial^2}{\partial x_i\partial x_j}u\mbox{ in }U\times[0,T];\\
&u(x,0)\geq 0 \mbox{ on } (\partial U\times [0,T])\cup (U\times \{0\})
\end{align*}
where $a^{ij}$ are smooth, bounded functions defined on $U\times [0,T]$ and $(a^{ij})$ is uniformly elliptic, say $\Lambda^{-1}|\xi|^2\leq a^{ij}\xi_i\xi_j\leq \Lambda |\xi|^2$ for all $\xi \in \mathbb{R}^n$. Moreover, if $u$ satisfies
\[
|u|\leq Ae^{a|x|^2}
\]
for some constants $A$, $a$. Then we have
\[
\max_{U\times[0,T]}u=\max_{(\partial U\times [0,T])\cup (U\times \{0\})}u
\]
\end{lemma}
\begin{proof}  First of all, we assume
\[
4aT<\Lambda^{-1}.
\]
So there exists $\varepsilon>0$ such that
\[
4(a+\varepsilon)(T+\varepsilon)<\lambda
\]

 The function
\begin{align*}
\varphi(x,t)=\mu (T+\varepsilon-t)^{-\frac{n}{2}}e^{\frac{ \Lambda^{-1}|x|^2}{4(T+\varepsilon-t)}}
\end{align*}
satisfies
\begin{align*}
&\frac{\partial}{\partial t}\varphi- a^{ij}(x,t) \frac{\partial^2}{\partial x_i\partial x_j}\varphi
\\
=&(\frac{n}{2}+\frac{\Lambda^{-1} |x|^2}{4(T-t)}-a^{ij}x_ix_j\frac{\Lambda^{-2}}{4(T+\varepsilon-t)}-\sum_{i}a^{ii}\frac{\Lambda^{-1}}{2})\mu (T+\varepsilon-t)^{-\frac{n}{2}+1}e^{-\frac{\Lambda^{-1}|x|^2}{4(T+\varepsilon-t)}}\\
\geq& 0
\end{align*}
for all $\mu\geq 0$.\\

Now, let $\nu=u-\varphi$. We have
\begin{align*}
\frac{\partial}{\partial t}\nu- a_{ij}(x,t) \frac{\partial^2}{\partial x_i\partial x_j}\nu
\leq 0.
\end{align*} 
for all $(x,t)\in U\times [0,T]$. We also have
\[
\max_{U\times \{0\}}\nu \leq \max_{U\times \{0\}}u.
\]

Finally, let $U_r=B_r(0)\cap U$. We have for all $x$, $|x|=r$
\begin{align*}
\nu(x,t)&\leq u(x,t)-\mu (T+\varepsilon-t)^{-\frac{n}{2}}e^{\frac{\Lambda^{-1} |x|^2}{4(T+\varepsilon)}}\\
&\leq Ae^{a|x^2|}-\mu \varepsilon^{-\frac{n}{2}}e^{(a+\varepsilon) |x|^2}\\
&\leq 0
\end{align*}
for $r$ large enough and all $t\in[0,T]$. Therefore by the maximum principle on bounded domain [11], we have
\[
\max_{U_r\times[0,T]}\nu=\max_{(\partial U_r\times [0,T])\cup (U_r\times \{0\})}u.
\]
for all $r$ large. This implies
\[
\max_{U\times[0,T]}\nu=\max_{(\partial U\times [0,T])\cup (U\times \{0\})}u.
\]
by taking $r\rightarrow \infty$ and $\mu\rightarrow 0$.\\

If it is not the case that $4aT<\Lambda^{-1}$, we can apply the result above on $[0,T_1]$, $[T_1,2T_1]$,... such that $T_1\leq \frac{\Lambda^{-1}}{8a}$.
\end{proof}

Now we can prove our proposition
\begin{pro} 
$|v|\rightarrow 0$ as $r\rightarrow \infty$ on $t\in [t_0,t_0+T_0]$ 
\end{pro}
\begin{proof}
We change our coordinate by taking
\begin{align*}
s=\begin{cases}&(1-r^{-1})^{\frac{1}{2}}\mbox{ for } r\leq 2\\
                              &r \mbox{ for } r\geq 3\\
                              &\mbox{a smooth increasing function for } r\in(2,3).
                              \end{cases}
\end{align*}
So we can write
\[
L(u)=g^{11}(x,t)\partial_s\partial_su.
\]

Using the parametric method [11], we can find the fundamental solution  of $L$, say $\Phi(s,\xi, t, \tau)$. Define
\[
W(s,t)=\Phi \ast |\langle F, v\rangle|=\int_{t_0}^{t_0+t}\int_0^\infty \Phi(s,\xi,t, \tau)|\langle F, v\rangle (\xi, \tau)| .
\]
So we have $W\geq0$ on $(\partial U\times [t_0,t_0+T_0])\cup (U\times \{t_0\})$ and
\[
\frac{\partial}{\partial t}W= L(W)+|\langle F, v\rangle|.
\]

Since $|v|^2=0$ on $(\partial U\times [t_0,t_0+T])\cup (U\times \{t_0\})$ and 
\[
\frac{\partial}{\partial t}|v|^2= L(|v|^2)+\langle F, v\rangle,
\]
we have
\[
\frac{\partial}{\partial t}(W-|v|^2)- L(W-|v|^2)\geq 0
\]
and $(W-|v|^2)\geq 0$.  Therefore by lemma 5.6, we have $|v|^2\leq W$.\\

Because $|\langle F,v\rangle| $ is bounded and $|\langle F,v\rangle| \leq C s^{-3}$ as $s\rightarrow \infty$ (recall that  $|Rm(\dot{g})|\leq Cr^{-3}$), by the definition of $W$ and the estimates of fundamental solutions [11]
\[
|\Phi(s,\xi,t,\tau)|\leq C (t-\tau)^{-\frac{1}{2}}e^{\frac{-c|s-\xi|^2}{t-\tau}},
\]
we can get $W\leq C s^{-3} $. So $|v|^2\leq Cs^{-3}$ which implies $|v|\rightarrow 0$ as $r\rightarrow \infty$ on $t\in [t_0,t_0+T_0]$.
\end{proof}

\begin{remark}
We notice that
\begin{align*}
F&=2g^{\alpha\alpha}g_{ii}\bar{g}^{ii}\bar{R}_{i\alpha i\alpha}=2(g^{\alpha\alpha}-\bar{g}^{\alpha\alpha})g_{ii}\bar{g}^{ii}\bar{R}_{i\alpha i\alpha}+2g_{ii}\bar{g}^{ii}\bar{R}_{ii}\\
&=2g^{\alpha\alpha}\bar{g}^{\alpha\alpha}v_{\alpha\alpha}g_{ii}\bar{g}^{ii}\bar{R}_{i\alpha i\alpha}+2g_{ii}\bar{g}^{ii}\bar{R}_{ii}.
\end{align*}

Now we fix a $t\in [0,T_0]$. By proposition 4.6, for any $l\in \mathbb{N}$, there exists a constant $C_l>0$ such that $|\bar{R}_{ii}|\leq C_l s^{-l}$. Since we have $|v|\leq Cs^{-\frac{3}{2}}$, so we have $|\langle F, v\rangle|\leq C s^{-6}$. This implies $ W\leq C s^{-6}$, which means $|v|\leq C s^{-3}$. Therefore, we can prove inductively that for any $l\in \mathbb{N}$, there exists a constant $C_l>0$ such that $|v|\leq C_l s^{-l}$. Now we use $p$-coordinate, we will have $|v|(p)\leq C_l (1-p)^{l}$ for all $l\in\mathbb{N}$.
\end{remark}
 In fact, we can prove the following lemma:
\begin{lemma}
Let $v$ be a solution of equation (5.1) and $T_0$ be the constant given by proposition 5.7. Then  for any $t\in [t_0,t_0+T_0]$ fixed, we have
\begin{align}
(|\bar{\nabla}^{(l)}v|\sqrt{\bar{g}})(p,t)\rightarrow 0
\end{align}
as $p\rightarrow 1$ for all $l\in \mathbb{N}$.
\end{lemma}

For the proof of this lemma, see appendix.

\subsection{Energy estimate and improved regularity:} We want to achieve two goals in this subsection. First, we shall prove that the Ricci-de Turck flow will stay in a estimate cone. Second, we want to show that the deviation of the Ricci-de Turck flow from initial data $\bar{g}$ will satisfy (\ref{cod1}). All these facts rely on the energy estimate of parabolic PDEs.\\

We use $|\cdot|_0$, $|\cdot|_\varepsilon$,  and $|\cdot|$ to denote the norms of tenser w.r.t $\dot{g}$, $\bar{g}$ and $g$. Also we define $\|w\|^{(0)}=(\int_M |w|_0^2 dVol^{(0)})^{\frac{1}{2}}$ and $\|w\|^{(\varepsilon)}_2=(\int_M |w|_{\varepsilon}^2 dVol^{(\varepsilon)})^{\frac{1}{2}}$.\\

Moreover, we introduce several norms for matrix valued functions on the $M\times [t_0, t_0+T]$.\\

\begin{definition}
Suppose that $(x,t)\mapsto k=k(x,t)$ defines a section over $M\times [t_0,t_0+T]$ of the bundle $Sym^2(T^*M)$. We define
\[
\|k\|^{(0)}_{L^2([t_0,t_0+T];L^2(M))}=(\int_{t_0}^{t_0+T}\int_M |k|^2_{0}(t) dt)^{\frac{1}{2}};
\] 
\[
\|k\|^{(0)}_{L^{\infty}([t_0,t_0+T];L^2(M))}=\sup_{s\in[t_0,t_0+T]}(\int_M |k|^2_{0}(s) )^{\frac{1}{2}}
\]
and
\[
\|k\|^{(0)}_{H^j([t_0,t_0+T];H^k(M))}=\sum_{p\leq j; q\leq k}(\int_{t_0}^{t_0+T}\int_M |(\partial_t)^p\dot{\nabla}^{(q)} k|^2_0(t) dt)^{\frac{1}{2}}.
\]
We can define the corresponding $\varepsilon$-version of these norms by replacing $\dot{g}$-norm by $\bar{g}$-norm and $\dot{\nabla}$ by $\bar{\nabla}$.
\end{definition}

Finally, we define the following notations.

\begin{definition}
Suppose that $(x,t)\mapsto A=A(x,t)$ defines a section over $M\times [s,s']$ of the bundle $Sym^2(T^*M)$ for some closed interval $[s,s']\in (-\infty,0]$ and $n\in \mathbb{N}$. We denote $A=O(\delta^n)$ iff $|A|(t)\leq C \delta^n(t)$ for some universal constant $C$ and all $t\in[s,s']$.
\end{definition}

Now we can prove the following theorem.
\begin{theorem}
There are universal constants $\bar{T}>0$, $C_0>0$ and $\iota>0$ such that if we fix  $\varepsilon\leq \iota$ and $T\leq \bar{T}$, then
\[
\|w\|^{(0)}_{2}(t), \|\bar{\nabla}w\|^{(0)}_{2}(t)\leq C_0T^{\frac{1}{2}} (\varepsilon+\varepsilon\delta(t)+\delta(t)^2+\varepsilon^2)
\]
for all $t\in[t_0,t_0+T]$.
\end{theorem}
\begin{proof} First, by remark 4.12, we can choose $\iota$ small such that $|Rm(\bar{g})|_\varepsilon\leq 2|Rm(g_0)|$ for all $\varepsilon\leq \iota$. Now take $\rho=\min\{\frac{1}{4},\frac{1}{8C_1}\}$ and $k_0=3|Rm(g_0)|$ in theorem 5.4, we will have $T_1(\rho,k_0)$ such that
\[
|g(x,t)-\bar{g}|_\varepsilon\leq \rho
\]
for all $t\in [t_0,t_0+T_1]$. Here we do some estimates of $|g-\bar{g}|_\varepsilon$ and $|g-\dot{g}|_0$, which can help us understand the different between norms $|\cdot|$, $|\cdot|_0$ and $|\cdot|_\varepsilon$. We first notice that our solution $g$ is defined on a small interval $[t_0,t_0+T_1]$ such that $|g-\bar{g}|_\varepsilon \leq \frac{1}{4}$. Secondly, recall that we have estimate (4.4) and $\rho \leq \frac{1}{8C_1}$. Combining all these estimates, we have $|\dot{g}-g|_0\leq \frac{1}{2}$. Therefore, all these norms are equivalent.\\

By theorem 5.5, there exists $T_2(k_0)$ such that $\sup_{M\times [t_0,t_0+T_2(k_0)]} |\bar{\nabla }g|\leq C(k_0)$ and $\sup_{M\times [t_0,t_0+T_2(k_0)]} |\partial_tg|\leq C(k_0)$.
We can also find a $T_3(k_0)$ which is given by theorem 4.9. which is the length of the time interval of the existence.\\

Finally, we take $T_0$ as the one in proposition 5.7. We set $\bar{T}= \min\{T_0,T_1,T_2, T_3\}$. Now we fix a $T\leq \bar{T}$.\\

Equation (5.2) tells us that
\begin{align*}
\frac{\partial}{\partial t}w_{kk}=g^{\alpha\alpha}\bar{\nabla}_{\alpha}\bar{\nabla}_{\alpha} w_{kk}&+(g^{\alpha\alpha}\bar{\nabla}_{\alpha}\bar{\nabla}_{\alpha}-\dot{g}^{11}\dot{\nabla}_1\dot{\nabla}_1)(\delta(t)-\varepsilon) h_{kk}\\
&+ (A_{kk}-2(\delta(t)-\varepsilon)\dot{R}^{i\mbox{ }i\mbox{ }}_{\mbox{ }k\mbox{ }k}h_{ii})-\lambda \varepsilon h_{kk}
\end{align*}
for $k=0,1,2,3$. If we set $B_{kk}= (g^{\alpha\alpha}\bar{\nabla}_\alpha\bar{\nabla}_\alpha-\dot{g}^{\alpha\alpha}\dot{\nabla}_\alpha\dot{\nabla}_\alpha)\delta h$ and $C_{kk}=(A_{kk}-2(\delta(t)-\varepsilon)\dot{R}^{i\mbox{ }i\mbox{ }}_{\mbox{ }k\mbox{ }k}h_{ii})$, we can write this equation as
\begin{align}\label{peq}
\frac{\partial}{\partial t}w_{kk}=g^{\alpha\alpha}\bar{\nabla}_{\alpha}\bar{\nabla}_{\alpha} w_{kk} +B_{kk}+C_{kk}-\lambda \varepsilon h_{kk}.
\end{align}
We will estimate each term of equation (\ref{peq}) in the following paragraphs.\\

Here we start with proving $\|B_{kk}\|^{(\varepsilon)}_{2}\leq C(\varepsilon\delta(t)+\varepsilon^2)$. Since we have
\begin{align*}
g^{\alpha\alpha}\bar{\nabla}_{\alpha}\bar{\nabla}_{\alpha}(\delta(t)-\varepsilon) h_{kk}=g^{\alpha\alpha}\partial_{\alpha}(\bar{\nabla}_{\alpha}(\delta(t)-\varepsilon) h_{kk})
&-g^{\alpha\alpha}\bar{\Gamma}^{\gamma}_{\alpha \alpha}(\bar{\nabla}_{\gamma} (\delta(t)-\varepsilon) h_{kk})\\
&-2g^{\alpha\alpha}\bar{\Gamma}^k_{\alpha k}(\bar{\nabla}_1 (\delta(t)-\varepsilon) h_{kk})
\end{align*}
and
\begin{align*}
\dot{g}^{\alpha\alpha}\dot{\nabla}_{\alpha}\dot{\nabla}_{\alpha}((\delta(t)-\varepsilon) h)_{kk}=\dot{g}^{\alpha\alpha}\partial_{\alpha}(\dot{\nabla}_{\alpha}(\delta(t)-\varepsilon) h_{kk})
&-\dot{g}^{\alpha\alpha}\dot{\Gamma}^{\gamma}_{\alpha\alpha}(\dot{\nabla}_{\gamma} (\delta(t)-\varepsilon) h_{kk})\\
&-2\dot{g}^{\alpha\alpha}\dot{\Gamma}^k_{\alpha k}(\dot{\nabla}_{\alpha} (\delta(t)-\varepsilon) h_{kk}),
\end{align*}
the $B_{kk}$ can be expressed as
\begin{align*}
(g^{\alpha\alpha}\bar{\nabla}_{\alpha}\bar{\nabla}_{\alpha}-\dot{g}^{\alpha\alpha}\dot{\nabla}_{\alpha}\dot{\nabla}_{\alpha})(\delta(t)-\varepsilon) h(x)_{kk}&\\
=(\dot{g}^{\alpha\alpha}-g^{\alpha\alpha})\partial_{\alpha}(\dot{\nabla}_{\alpha}((\delta(t)-\varepsilon) h)_{kk})-g^{\alpha\alpha}&\partial_{\alpha}(\bar{\nabla}_{\alpha}-\dot{\nabla}_{\alpha})((\delta(t)-\varepsilon) h)_{kk}\\
-(\dot{g}^{\alpha\alpha}-g^{\alpha\alpha})\dot{\Gamma}^{\gamma}_{\alpha\alpha}(\dot{\nabla}_{\gamma}((\delta(t)-\varepsilon) h)_{kk})-&g^{\alpha\alpha}(\bar{\Gamma}_{\alpha\alpha}^{\gamma}-\dot{\Gamma}_{\alpha\alpha}^{\gamma})(\dot{\nabla}_{\gamma} (\delta(t)-\varepsilon) h)_{kk}\\
-2(\dot{g}^{\alpha\alpha}-g^{\alpha\alpha})\dot{\Gamma}^k_{\alpha k}(\dot{\nabla}_{\alpha}((\delta(t)-\varepsilon) h)_{kk})&-2g^{\alpha\alpha}(\bar{\Gamma}_{\alpha k}^k-\dot{\Gamma}_{\alpha k}^k)(\dot{\nabla}_{\alpha} (\delta(t)-\varepsilon) h)_{kk}\\
&+g^{\alpha\alpha}\bar{\Gamma}_{\alpha\alpha}^{\gamma}(\bar{\nabla}_{\gamma}\dot{\nabla}_{\gamma})(\delta(t)-\varepsilon) h_{kk}\\
&+2g^{\alpha\alpha}\bar{\Gamma}_{\alpha k}^k(\bar{\nabla}_{\alpha}\dot{\nabla}_1)(\delta(t)-\varepsilon) h_{kk}
\end{align*}
\begin{align*}
=\mbox{ }&(\dot{g}^{\alpha\alpha}-g^{\alpha\alpha})\dot{\nabla}_1\dot{\nabla}_1((\delta(t)-\varepsilon) h)_{kk}+g^{\alpha\alpha}\bar{\nabla}_{\alpha}(\dot{\nabla}_{\alpha}-\bar{\nabla}_{\alpha})((\delta(t)-\varepsilon) h)_{kk}\\
&-g^{\alpha\alpha}(\bar{\Gamma}_{\alpha\alpha}^{\gamma}-\dot{\Gamma}_{\alpha\alpha}^{\gamma})(\bar{\nabla}_{\gamma} (\delta(t)-\varepsilon) h)_{kk}-2g^{\alpha\alpha}(\bar{\Gamma}_{\alpha k}^k-\dot{\Gamma}_{\alpha k}^k)(\bar{\nabla}_{\alpha} (\delta(t)-\varepsilon) h)_{kk}
\end{align*}
Therefore we prove
\[
|B_{kk}|_\varepsilon\leq C (\varepsilon \delta(t)+\varepsilon^2)(|h|_\varepsilon+|\nabla h|_\varepsilon)
\]
by using inequalities (4.4), (4.5) and (4.6). We get our estimate after integration.
\\

Next, we estimate $C_{kk}=(A_{kk}-2(\delta(t)-\varepsilon)\dot{R}^{i\mbox{ }i\mbox{ }}_{\mbox{ }k\mbox{ }k}h_{ii})$. First, we consider $A_{kk}$. When $k\neq 1$, we have
\begin{equation}\label{kn1}
\begin{array}{rl}
A_{kk}&=-2g^{ii}g_{kk}\bar{g}^{kk}\bar{R}_{ikik}\\
&=2(\bar{g}^{ii}-g^{ii})g_{kk}\bar{g}^{kk}\bar{R}_{ikik}
-2g_{kk}\bar{g}^{kk}\bar{R}_{kk}\\
&= 2g^{ii}\bar{g}^{ii}(w_{ii}+\delta(t)h_{ii}-\varepsilon h_{ii})g_{kk}\bar{g}^{kk}\bar{R}_{ikik}-2g_{kk}\bar{g}^{kk}\bar{R}_{kk}\\
&=2g^{ii}\bar{g}^{ii}(w_{ii}+\delta(t)h_{ii}-\varepsilon h_{ii})(w_{kk}+\delta(t)h_{kk}-\varepsilon  h_{kk})\bar{g}^{kk}\bar{R}_{ikik}\\
&\mbox{ }+2g^{ii}\bar{g}^{ii}(w_{ii}+\delta(t)h_{ii}-\varepsilon h_{ii})\bar{R}_{ikik}-2g_{kk}\bar{g}^{kk}\bar{R}_{kk}.\\
\end{array}
\end{equation}

Here we have
\begin{align*}
&2g^{ii}\bar{g}^{ii}(w_{ii}+\delta(t)h_{ii}-\varepsilon h_{ii})(w_{kk}+\delta(t)h_{kk}-\varepsilon h_{kk})\bar{g}^{kk}\bar{R}_{ikik}]\\
=&\sum_{i}(b^{ii} w_{ii}w_{kk}+ c^{ii}_{\mbox{ }\mbox{ }kk} w_{ii})+(d\delta^2(t)+e\varepsilon\delta(t)+f\varepsilon^2)_{kk}
\end{align*}
for some $b$ bounded, $c=O(\delta)$ and $d$, $e$, $f$ in $L^{\infty}([t_0,T_0+T];L^2(M))$. So we have
\begin{align*}
C_{kk}= &\sum_{i}(b^{ii} w_{ii}w_{kk}+ c^{ii}_{\mbox{ }\mbox{ }kk} w_{ii})+(d\delta^2(t)+e\varepsilon\delta(t)+f\varepsilon^2)_{kk}\\
&\mbox{ }+2g^{ii}\bar{g}^{ii}(w_{ii}+\delta(t)h_{ii}-\varepsilon h_{ii})\bar{R}_{ikik}-2(\delta(t)-\varepsilon)\dot{R}^{i\mbox{ }i\mbox{ }}_{\mbox{ }k\mbox{ }k}h_{ii}-2g_{kk}\bar{g}^{kk}\bar{R}_{kk}.\\
\end{align*}

 Therefore, we have to estimate the following terms:
\[
2g^{ii}\bar{g}^{ii}(w_{ii}+\delta(t)h_{ii}-\varepsilon h_{ii})\bar{R}_{ikik}-2\dot{R}^{i\mbox{ }i\mbox{ }}_{\mbox{ }k\mbox{ }k}(\delta(t)-\varepsilon)h_{ii}-2g_{kk}\bar{g}^{kk}\bar{R}_{kk}.
\]
First of all, we have $g^{ii}\bar{g}^{ii}w_{ii}\bar{R}_{ikik}=\bar{R}^{i\mbox{ }i}_{\mbox{ }k\mbox{ }k}w_{ii}+c^{ii}_{\mbox{ }\mbox{ }kk}w_{ii}$ for some $c=O(\delta)$. Secondly,
 since we have
\begin{align*}
&2g^{ii}\bar{g}^{ii}\bar{R}_{ikik}(\delta(t)-\varepsilon)h_{ii}-2\dot{R}^{i\mbox{ }i\mbox{ }}_{\mbox{ }k\mbox{ }k}(\delta(t)-\varepsilon)h_{ii}\\
=& 2(g^{ii}\bar{R}^{\mbox{ }\mbox{ }i}_{ik\mbox{ }k}(\delta(t)-\varepsilon) h_{ii}-\dot{g}^{ii}\dot{R}^{\mbox{ }\mbox{ }i}_{ik\mbox{ }k}(\delta(t)-\varepsilon) h_{ii})\\
=& 2(g^{ii}\dot{R}^{\mbox{ }\mbox{ }i}_{ik\mbox{ }k}(\delta(t)-\varepsilon) h_{ii}-\dot{g}^{ii}\dot{R}^{\mbox{ }\mbox{ }i}_{ik\mbox{ }k}(\delta(t)-\varepsilon) h_{ii})+2g^{ii}(\bar{R}^{\mbox{ }\mbox{ }i}_{ik\mbox{ }k}-\dot{R}^{\mbox{ }\mbox{ }i}_{ik\mbox{ }k})(\delta(t)-\varepsilon) h_{ii}\\
=& 2(g^{ii}-\dot{g}^{ii})\dot{R}^{\mbox{ }\mbox{ }i}_{ik\mbox{ }k}(\delta(t)-\varepsilon) h_{ii}+2g^{ii}(\bar{R}^{\mbox{ }\mbox{ }i}_{ik\mbox{ }k}-\dot{R}^{\mbox{ }\mbox{ }i}_{ik\mbox{ }k})(\delta(t)-\varepsilon) h_{ii}\\
=&2g^{ii}\dot{g}^{ii}(w_{ii}+(\delta(t)-\varepsilon) h_{ii})\dot{R}^{\mbox{ }\mbox{ }i}_{ik\mbox{ }k}(\delta(t)-\varepsilon) h_{ii}+2g^{ii}(\bar{R}^{\mbox{ }\mbox{ }i}_{ik\mbox{ }k}-\dot{R}^{\mbox{ }\mbox{ }i}_{ik\mbox{ }k})(\delta(t)-\varepsilon) h_{ii}
\end{align*}
and $|\bar{R}^{\mbox{ }\mbox{ }i}_{ik\mbox{ }k}-\dot{R}^{\mbox{ }\mbox{ }i}_{ik\mbox{ }k}|\leq C \varepsilon$, we have
\[
2g^{ii}\bar{g}^{ii}(w_{ii}+\delta(t)h_{ii}-\varepsilon h_{ii})\bar{R}_{ikik}-2\dot{R}^{i\mbox{ }i\mbox{ }}_{\mbox{ }k\mbox{ }k}(\delta(t)-\varepsilon)h_{ii}=\sum_{i}c^{ii}_{\mbox{ }\mbox{ }kk} w_{ii}+(d\delta^2(t)+e\varepsilon\delta(t)+f\varepsilon^2)_{kk}
\]
for some $c=O(\delta)$ and $d$, $e$, $f$ in $L^{\infty}([t_0,t_0+T];L^2(M))$. Finally we consider the term $2g_{kk}\bar{g}^{kk}\bar{R}_{kk}$. Since $\dot{R}_{kk}=0$ (Ricci flat), we have $|\bar{R}_{kk}|\leq |\bar{R}_{kk}-\dot{R}_{kk}|\leq \varepsilon$. So we can easily get the estimate of this term:
\[
2g_{kk}\bar{g}^{kk}\bar{R}_{kk}=cw_{kk}+\varepsilon\bar{f}_{kk}.
\]
for some $c=O(\delta)$ and $\bar{f}\in L^{\infty}([t_0,T_0+T];L^2(M))$. Therefore we can write $C_{kk}$ as following
\[
C_{kk}=2\bar{R}^{i\mbox{ }i}_{\mbox{ }k\mbox{ }k}w_{ii}+\sum_{i}(b^{ii} w_{ii}w_{kk}+ c^{ii}_{\mbox{ }\mbox{ }kk} w_{ii})+(d\delta^2(t)+e\varepsilon\delta(t)+f\varepsilon^2+\bar{f}\varepsilon)_{kk}
\]
for some $b$ bounded, $c=O(\delta)$ and $d$, $e$, $f$, $\bar{f}$ in $L^{\infty}([t_0,t_0+T];L^2(M))$.\\

For $k=1$, we will have an extra term 
\[
\frac{1}{2}g^{ll}g^{ll}((\bar{\nabla}_1w_{ll}+\bar{\nabla}_1\delta h_{ll})(\bar{\nabla}_1w_{ll}+\bar{\nabla}_1\delta h_{ll}))
\]
which can be expressed as 
\[
\sum_{l}((\bar{b}^{ll}\bar{\nabla}_1w_{ll})^2+\bar{c}^{ll}_{\mbox{ }\mbox{ }1}\bar{\nabla}_1w_{ll})+(d\delta^2(t))_{11}
\]
with $\bar{b}$ bounded, $\bar{c}=O(\delta)$ and $d$ in $L^{\infty}([t_0,t_0+T];L^2(M))$. We finish the estimate of $C_{kk}$ here.\\

Finally, $\lambda \varepsilon h_{kk}$ is bounded by $(\varepsilon\bar{f})_{kk}$ for some $\bar{f}\in L^{\infty}([t_0,t_0+T];L^2(M))$.\\

We can summarize that $w=(w_{kk})_{k=0,1,2,3}$ will satisfy a system of parabolic equations
\begin{align}
\frac{\partial}{\partial t}w = E( w) + F;
\end{align}
with $w=0 \mbox{ on } M\times\{t_0\}$ and $\partial M\times [t_0,t_0+T]$, where $F=d\delta^2(t)+e\varepsilon\delta(t)+f\varepsilon^2+\bar{f}\varepsilon$ and
\begin{align*}
E( w_{kk})=g^{\alpha\alpha}\bar{\nabla}_\alpha\bar{\nabla}_\alpha w_{kk} &+2\bar{R}^{i\mbox{ }i}_{\mbox{ }k\mbox{ }k}w_{ii}\\
&+\delta_{k1}((\bar{b}^{ll} \bar{\nabla}_1w_{ll})^2+\bar{c}^{ll}_{\mbox{ }\mbox{ }1} \bar{\nabla}_1w_{ll})+ \sum_{i}b^{ii}w_{ii}w_{kk}+c^{ii}_{\mbox{ }\mbox{ }kk}w_{ii}
\end{align*}
($\delta_{k1}$ is the Kronecker's delta).\\

 Because $T\leq \min\{T_1,T_2\}$, this implies $w_{kk}$, $\bar{\nabla}w_{kk}$ are bounded. So $w$ can be regarded as a solution of linear equations
\begin{align}
\frac{\partial}{\partial t}\mathbf{v} = L \mathbf{v} + F
\end{align}
where
\[
L v_k=g^{\alpha\alpha}\bar{\nabla}_\alpha\bar{\nabla}_\alpha v_{k} +\delta_{k1}( (\bar{b}^{ll})^2\bar{\nabla}_1w_{ll}\bar{\nabla}_1v_l+\bar{c}^{ll}_{\mbox{ }\mbox{ }1} \bar{\nabla}_1v_l)+ \sum_{i}(b^{ii}w_{ii})v_k + c^{ii}_{kk}v_i+2\bar{R}^{i\mbox{ }i}_{\mbox{ }k\mbox{ }k}v_{i}
\]
and $\|F\|^{(\varepsilon)}_{2}\leq C (\varepsilon+ \varepsilon\delta(t)+\delta(t)^2+\varepsilon^2)$. Therefore, by lemma 5.9, we can apply the energy estimate of linear parabolic equation from [6] on $M$. We will have
\begin{align*}
(\| w \|^{(\varepsilon)}_2+\|\bar{\nabla} w \|^{(\varepsilon)}_2)(t)&+\|\partial_t w \|_{L^2([t_0,t_0+T];L^2(M))}^{(\varepsilon)}\\&\leq C_0 (\|F\|^{(\varepsilon)}_{L^2([t_0,t_0+T];L^2(M))}+\|w\|^{(\varepsilon)}_{H^1(M)}(t_0))\\
&\leq C_0T^{\frac{1}{2}}(\varepsilon+\varepsilon\delta(t_0+T)+\delta(t_0+T)^2+\varepsilon^2)
\end{align*}
for all $t\in [t_0,t_0+T]$. So
\begin{align*}
(\| w \|^{(\varepsilon)}_2+\|\bar{\nabla} w \|^{(\varepsilon)}_2)(t)& \leq C_0T^{\frac{1}{2}} e^{-\lambda T} (\varepsilon+\varepsilon\delta(t_0)+\delta(t_0)^2+\varepsilon^2)\\
&\leq C_0T^{\frac{1}{2}}e^{-\lambda \bar{T}} (\varepsilon+\varepsilon\delta(t)+\delta(t)^2+\varepsilon^2).
\end{align*}
for all $t\in [t_0,t_0+T]$. Because $\|\cdot\|^{(0)}$ and $\|\cdot\|^{(\varepsilon)}$ are equivalent, we have completed the proof.
\end{proof}
\begin{remark}
Here we notice that $\varepsilon \bar{f}_{kk}=2\bar{R}_{kk}-\lambda \varepsilon h_{kk}=(\frac{\partial}{\partial \delta}R_{kk}(\dot{g}+\delta h)|_{\delta=0}-\lambda h_{kk})\varepsilon^2+O(\varepsilon^2)$. In section 3, we already know that 
\[
\frac{\partial}{\partial \delta}R_{kk}(\dot{g}+\delta h)|_{\delta=0}=-(\Delta_L h)_{kk}-\nabla_k\nabla_k H-2\nabla_k(\zeta h)_k
\]
for all $k$. However, since $\dot{g}$ and $h$ are radially symmetric and diagonal, we have
$\nabla_k\nabla_k H+2\nabla_k(\zeta h)_k=0$ for all $k\neq 1$. This means that $\varepsilon\bar{f}_{kk}=O(\varepsilon^2)$ for $k\neq 1$.
\end{remark}
We have the following corollary immediately.
\begin{cor}
There are universal constants $\bar{T}>0$, $C_0>0$ and $\iota$ such that if we fix  $\varepsilon\leq \iota$ and $T\leq\bar{T}$, then
\[
\|w\|^{(0)}_{2}(t), \|\bar{\nabla}w\|^{(0)}_{2}(t)\leq C_0T^{\frac{1}{2}} \delta(t) 
\]
for $\varepsilon$ is small enough and all $t\in[t_0,t_0+T]$. 
\end{cor}

We can improve the regularity and get the following estimate.

\begin{theorem}
Let $w$ be defined as above. We have
\begin{align}
|w|_0(t), |\bar{\nabla}w|_0 (t), \mbox{ } |\bar{\nabla}\bar{\nabla}w|_{0}(t)\mbox{ and } |\bar{\nabla}^{(3)}w|_{0}(t) \leq C \delta(t)
\end{align}
for all $x\in M$, $t\in [t_0,t_0+T]$, $T\leq \bar{T}$ and a universal constant $C$.
\end{theorem}

\begin{proof} Because $\|\cdot\|^{(0)}$ and $\|\cdot\|^{(\varepsilon)}$ are equivalent, we can just prove these inequalities using the $\varepsilon$-norm. To get the improved regularity, we notice that $|h|$, $|\nabla h| \in L^2(M)\cap L^{\infty}(M)$. The energy estimate tells us that
\begin{align}
\|w_t\|^{(\varepsilon)}_{L^2([t_0, t_0+T];L^2(M))}\leq C_0(\varepsilon+\varepsilon\delta(t_0+T)+\delta(t_0+T)^2+\varepsilon^2)
\end{align}
which implies $\partial_td$, $\partial_te$, $\partial_tf$ and $\partial_t\bar{f}$ are in $L^2([t_0, t_0+T];M)$. Therefore we have $F\in H^1([t_0,t_0+T];L^2(M))$. To get the improved regularity, we consider the equation for $w_t$ by differentiate equation (5.5) with respect to $t$ on the both sides. We will have a parabolic equation of the form 
\begin{align}
\frac{\partial}{\partial t}w_t=L'(w_t)+K+F_t
\end{align}
where $|K|_{\varepsilon}\leq O(\delta)(|\bar{\nabla}\bar{\nabla}w|_\varepsilon+|\bar{\nabla}w|_\varepsilon+|w|_\varepsilon)$. By lemma 5.9, we can apply the energy estimate 
\begin{align*}
(\| w_t \|^{(\varepsilon)}_2+\| \bar{\nabla} w_t\|^{(\varepsilon)}_2)(t)\leq C_0(\|K+F_t\|^{(\varepsilon)}_{L^1([t_0,t_0+T];L^2(M))}+\|w\|^{(\varepsilon)}_{H^3(M)}(t_0))
\end{align*}
which implies that $w_t$ and $\bar{\nabla}w_t$ are also in $L^2(M)$. By equation (5.5) and proposition 4.2, the sup norms of $w$, $\bar{\nabla}w$ and $w_t$ are bounded. Thus we have proved that the asserted bound for the sup norm of the boundedness of sup norm of $\bar{\nabla}\bar{\nabla}w$.\\

To estimate $\bar{\nabla}^{(3)}w$, we should prove that $\bar{\nabla}w_t$ is bounded first. We consider the equation $\frac{\partial}{\partial t}w_t=L'(w_t)+F_t$ again. We use the first derivative estimate in [2] (Theorem 6.1 in VII) to get the sup norm of $\bar{\nabla}w_t$. We use this bound to bound $|\bar{\nabla}^{(3)}w|$ from the equation $\partial_t\bar{\nabla}w=\bar{\nabla}L w+\bar{\nabla}F$ and the bounds for $w$, $\bar{\nabla}w$, $\bar{\nabla}\bar{\nabla}w$ and $\bar{\nabla}w_t$.
\end{proof}

\subsection{Inductive estimates}
In this subsection, we will prove a lemma which is essential in our proof of the long-time existence.
\begin{lemma}Let $\iota$ be the constant given by theorem 5.12. There exist universal constants $M_1,M_2,M_3,M_4,T>0$ and $\bar{N}\leq 0$ such that if $w$ is the solution of equation (5.8) defined on $M\times [t_0,\bar{N}]$ with $\varepsilon\leq\iota$ and  \\
1. $w$ satisfies
\begin{align}
&\|w\|^{(\varepsilon)}_{2}(t)\leq M_1\delta(t),\nonumber\\
 &\|\bar{\nabla}w\|^{(\varepsilon)}_{2}(t)\leq M_2\delta(t), \\
 &\|w_t\|^{(\varepsilon)}_{2}(t)\leq M_3\delta(t), \nonumber\\
 &\|\bar{\nabla}w_t\|^{(\varepsilon)}_{2}(t)\leq M_4\delta(t), \nonumber
\end{align}
\ \ \ for all $t\in [t_0,t_0+(k-1)T]$ and\\
2. $t_0+kT<\bar{N}$,\\
then the inequalities (5.12) hold for $t\in [t_0,t_0+kT]$.
\end{lemma}
\begin{proof} By equation (5.8), $w$ satisfies
\begin{align}
\frac{\partial}{\partial t}w_{kk}=g^{\alpha\alpha}\bar{\nabla}_{\alpha}\bar{\nabla}_{\alpha}w_{kk} &+2\bar{R}^{i\mbox{ }i}_{\mbox{ }k\mbox{ }k}w_{ii}\\
&+\delta_{k1}((b^{kk} \bar{\nabla}_1w_{kk})^2+\bar{c}^{ii}_{\mbox{ }\mbox{ }1} \bar{\nabla}_1w_{ii})\nonumber\\&+ \sum_{i}b^{ii}w_{ii}w_{kk}+c^{ii}_{\mbox{ }\mbox{ }kk}w_{ii}\nonumber\\
&+d_{kk}\delta^2(t)+e_{kk}\varepsilon\delta(t)+f_{kk}\varepsilon^2+\bar{f}_{kk}\varepsilon\nonumber
\end{align}
with $b$ bounded, $\bar{c}=O(\delta)$, $c=O(\delta)$ and $d$, $e$, $f$, $\bar{f}$ in $L^{\infty}([t_0,t_0+(k-1)T];L^2(M))$ for all $t\in[t_0,t_0+(k-1)T]$.\\

Now let $\bar{T}$ be the constant given by theorem 5.12. We fix the constants $M_1$, $M_2$, $M_3$, $M_4$, $T<\bar{T}$ and $\bar{N}\leq \min_{i=1,2,3,4}\{\frac{1}{-\lambda}\log(M_i^{-1})\}$ which will be specified later. This implies that $M_i\delta(t)\leq 1$ for all $i=1,2,3,4$ and $t\in[t_0,\bar{N}]$.\\

In what follows, $\bar{C}$ denotes a constant that depends only on $M_1$, $M_2$ and $M_3$; its precise value can be assumed to increase between each successive appearance. Similarly $\bar{D}$ denotes a constant that depending only on $M_1$, $M_2$, $M_3$ and $M_4$.\\

First of all, by (5.12) and $M_i\delta(t)\leq 1$ for all $t\in[t_0,t_0+(k-1)T]$, we can find a universal constant $V$ such that
\begin{align}
&\|d\|_2^{(\varepsilon)}(t),\|e\|_2^{(\varepsilon)}(t),\|f\|_2^{(\varepsilon)}(t),\|\bar{f}\|_2^{(\varepsilon)}(t),\nonumber\\
&\|d_t\|_2^{(\varepsilon)}(t),\|e_t\|_2^{(\varepsilon)}(t),\|f_t\|_2^{(\varepsilon)}(t),\|\bar{f}_t\|_2^{(\varepsilon)}(t)\leq V
\end{align}
for all $t\in [t_0,t_0+(k-1)T]$.\\

Secondly, by (5.12) and proposition 4.2, we have
\begin{align}
\sup_{x\in M}|\bar{\nabla}^{(n)}w|_{\varepsilon}(x,t)\leq C\delta(t)
\end{align}
for all $t\in [t_0,t_0+(k-1)T]$, $0\leq n\leq 3$ and a constant $C$ depending on $M_1,\cdots ,M_{n+1}$. Using (5.15) and the fact that $\bar{c}=O(\delta)$, $c=O(\delta)$, we have 
\begin{align}
\|\delta_{k1}(b^{kk} (\bar{\nabla}_1w_{kk})^2+\bar{c}^{ii}_{\mbox{ }\mbox{ }1} \bar{\nabla}_1w_{ii})+ \sum_{i}b^{ii}w_{ii}w_{kk}+c^{ii}_{\mbox{ }\mbox{ }kk}w_{ii}\|^{(\varepsilon)}_2(t)\leq \bar{C}\delta^2(t)
\end{align}
for all $t\in[t_0,t_0+(k-1)T]$. Therefore we have
\begin{align*}
\int_M \langle\frac{\partial}{\partial t}w, w\rangle \leq\int_M\langle g^{\alpha\alpha}\bar{\nabla}_{\alpha}\bar{\nabla}_{\alpha}w,w\rangle &+2\int_M\langle\bar{R}^{i\mbox{ }i}_{\mbox{ }k\mbox{ }k}w_{ii}, w_{kk}\rangle\\
&+\bar{C}\delta^2(t)\|w\|_2^{(\varepsilon)}\\
&+\int_M\langle d_{kk}\delta^2(t)+e_{kk}\varepsilon\delta(t)+f_{kk}\varepsilon^2+\bar{f}_{kk}\varepsilon, w\rangle\\
\leq\int_M\langle g^{\alpha\alpha}\bar{\nabla}_{\alpha}\bar{\nabla}_{\alpha}w,w\rangle &+2\int_M\langle\bar{R}^{i\mbox{ }i}_{\mbox{ }k\mbox{ }k}w_{ii}, w_{kk}\rangle\\
&+\bar{C}\delta^3(t)+3M_1V\delta^3(t)+M_1V\varepsilon\delta(t)
\end{align*}

Recall that $g-\bar{g}=w+(\delta(t)-\varepsilon)h$. Because 
\begin{align*}
g^{\alpha\alpha}\bar{\nabla}_{\alpha}\bar{\nabla}_{\alpha}w_{kk}&=\bar{\nabla}_{\alpha}(g^{\alpha\alpha}\bar{\nabla}_{\alpha}w_{kk})-(\bar{\nabla}_{\alpha}g^{\alpha\alpha})\bar{\nabla}_{\alpha}w_{kk}\\
&=\bar{\nabla}_{\alpha}(g^{\alpha\alpha}\bar{\nabla}_{\alpha}w_{kk})-(\bar{\nabla}_{\alpha}(g^{\alpha\alpha}-\bar{g}^{\alpha\alpha}))\bar{\nabla}_{\alpha}w_{kk},
\end{align*}
by (5.12), we do the integration by parts to get
\begin{align}
\int_M\langle g^{\alpha\alpha}\bar{\nabla}_{\alpha}\bar{\nabla}_{\alpha}w,w\rangle \leq -\int_M |\dot{\nabla}w|^2+ \bar{C} \delta^3(t).
\end{align}
We also have
\begin{align*}
2\int_M\langle\bar{R}^{i\mbox{ }i}_{\mbox{ }k\mbox{ }k}w_{ii}, w_{kk}\rangle&\leq 2\int_M \dot{R}^{ikik}w_{ii}w_{kk}+O(\varepsilon)\bar{C}\delta^2(t)\\
&\leq 2\int_M \dot{R}^{ikik}w_{ii}w_{kk}+ \bar{C}\delta^3(t)
\end{align*}
 by remark 4.12.\\

 Now we recall the definition of $\mathbf{a}$ and $\lambda$ in section 3.2. We can conclude that 
\begin{align*}
\frac{\partial}{\partial t}(\|w\|^{(\varepsilon)}_2)^2\leq -(\int_M |&\dot{\nabla}w|^2-2\dot{R}^{ikik}w_{ii}w_{kk})+\bar{C}\delta^3(t)
+3M_1V\delta^3(t)+M_1V\varepsilon\delta(t))\\
\leq -2\mathbf{a}(w)&+\bar{C}\delta^3(t)
+3M_1V\delta^3(t)+M_1V\varepsilon\delta(t))\\
\leq -2\lambda(\|w&\|^{(\varepsilon)}_2)^2+\bar{C}\delta^3(t)+3M_1V\delta^3(t)+M_1V\varepsilon\delta(t))\\
\leq -2\lambda(\|w&\|^{(\varepsilon)}_2)^2+(\bar{C}+3M_1V)\delta^3(t)+M_1V\varepsilon\delta(t))
\end{align*}
for all $t\in[t_0,t_0+(k-1)T]$.\\

 For $t\in [t_0+(k-1)T,t_0+kT]$, by using the energy estimate on equation (5.7), we have
\begin{align}
\|w\|^{(\varepsilon)}_2(t)+\|\bar{\nabla}w\|^{(\varepsilon)}_2(t)&\leq C_0(\|F\|^{(\varepsilon)}_{L^2([t_0+(k-1)T,t_0+kT];L^2(M))}+\|w\|_{H^1}^{(\varepsilon)}(t_0+(k-1)T))\\
&\leq (4C_0VT^{\frac{1}{2}}+C_0(M_1+M_2))\delta(t)\nonumber
\end{align}
where $C_0$ is the constant given by theorem 5.12. So if we replace (5.12) by (5.18), we will have the estimate
\begin{align*}
\frac{\partial}{\partial t}(\|w\|^{(\varepsilon)}_2)^2\leq &-2\lambda(\|w\|^{(\varepsilon)}_2)^2\\&+(\bar{C}+3(4C_0VT^{\frac{1}{2}}+C_0(M_1+M_2))V)\delta^3(t)+M_1V\varepsilon\delta(t)
\end{align*}
for all $t\in[t_0+(k-1)T,t_0+kT]$.\\

Therefore, if we choose $M_1= -2\lambda V$, $A=\bar{C}+2M_1V+3(4C_0VT^{\frac{1}{2}}+C_0(M_1+M_2))V$,
 we will have
\begin{align*}
\frac{\partial}{\partial t}&(\|w\|^{(\varepsilon)}_2)^2\leq -2\lambda(\|w\|^{(\varepsilon)}_2)^2+A\delta^3(t)-\frac{\lambda}{2} M^2_1\varepsilon\delta(t)
\end{align*}
for all $t\in [t_0,t_0+kT]$. Now by Grownwall's inequality, we have
\begin{align*}
(\|w\|^{(\varepsilon)}_2)^2&\leq e^{-2\lambda t}(\int_0^{t-t_0}A\delta^3(s)e^{2\lambda s}ds+\int_0^{t-t_0}-\frac{\lambda}{2} M^2_1\varepsilon\delta(s)e^{2\lambda s}ds)\\
&\leq \frac{2}{-\lambda}A\delta^3(t)+\frac{1}{2}M^2_1\delta^2(t)
\end{align*}
for all $t\in [t_0,t_0+kT]$. So
\[
(\|w\|^{(\varepsilon)}_2)^2 \leq M^2_1\delta^2(t)
\]
for all $t\in [t_0,t_0+kT]$ provided $t_0+kT\leq \frac{1}{-\lambda}\log(\frac{-\lambda M_1^2}{4A})$. We use $N_1$ in what follows to denote $\frac{1}{-\lambda}\log(\frac{-\lambda M_1^2}{4A})$.\\

Next, we define $M_2$. By multiplying both sides of (5.13) by $w_t$ then integrating both sides, it follows from Cauchy's inequality that
\begin{align*}
(\|w_t\|_2^{(\varepsilon)})^2(t)=\int_M\langle w_t,w_t\rangle \leq& \int_M\langle g^{\alpha\alpha}\bar{\nabla}_{\alpha}\bar{\nabla}_{\alpha}w,w_t\rangle +\int_M 2\bar{R}^{ikik}w_{ii}(w_t)_{kk} \\
&+\bar{C}M_2\delta^4(t)+\frac{1}{2}(\|w_t\|_2^{(\varepsilon)})^2(t)\\
&+ 4(\|d\|_2^{(\varepsilon)}\delta^2(t)+\|e\|_2^{(\varepsilon)}\varepsilon\delta(t)+\|f\|_2^{(\varepsilon)}\varepsilon^2+\|\bar{f}\|^{(\varepsilon)}_2\varepsilon)^2.
\end{align*}
for $t\in[t_0,t_0+(k-1)T]$. Using integration by parts and then integrating from $t=t_0$ to $t=(k-1)T$, we have
\begin{align*}
\frac{1}{2}(\|w_t\|_{L^2([t_0,t_0+(k-1)T];L^2(M))}^{(\varepsilon)})^2\leq& -(\|\bar{\nabla}w\|_2^{(\varepsilon)})^2(t_0+(k-1)T)\\
&+\int_M \bar{R}^{ikik}w_{ii}w_{kk}(t_0+(k-1)T)\\
&+\int_{t_0}^{t_0+(k-1)T}\bar{C}M_2\delta^4(s)ds\\
+ \int_{t_0}^{t_0+(k-1)T}4(\|d\|_2^{(\varepsilon)}\delta^2(s)&+\|e\|_2^{(\varepsilon)}\varepsilon\delta(s)+\|f\|_2^{(\varepsilon)}\varepsilon^2+\|\bar{f}\|^{(\varepsilon)}_2\varepsilon)^2ds.
\end{align*}
So
\begin{align}
&(\|w_t\|_{L^2([t_0,t_0+(k-1)T];L^2(M))}^{(\varepsilon)})^2+2(\|\bar{\nabla}w\|_2^{(\varepsilon)})^2(t_0+(k-1)T)\\
\leq &
4M_1^2\delta^2(t_0+(k-1)T)
+ \bar{C}\delta^3(t_0+(k-1)T)+|\lambda|^{-1}V^2\delta^2(t_0+(k-1)T)\nonumber.
\end{align}
\\

For $t\in[t_0+(k-1)T,t_0+kT]$, if we replace (5.12) by (5.18) and follow the computation to derive (5.19), we will have
\begin{align*}
&(\|w_t\|_{L^2([t_0+(k-1)T,t];L^2(M))}^{(\varepsilon)})^2+2(\|\bar{\nabla}w\|_2^{(\varepsilon)})^2(t)\\
\leq &2(\|\bar{\nabla}w\|_2^{(\varepsilon)})^2(t_0+(k-1)T)+ \int_{t_0+(k-1)T}^{t} \hat{C}\delta^2(s)ds\\
\leq &2(\|\bar{\nabla}w\|_2^{(\varepsilon)})^2(t_0+(k-1)T)+ \hat{C}T\delta^2(t)
\end{align*}
for  all $t\in[t_0+(k-1)T,t_0+kT]$ and some $\hat{C}$ depending on $M_1$ and $M_2$. By (5.19), we have
\begin{align*}
2(\|w_t\|_{L^2([t_0,t ];L^2(M))}^{(\varepsilon)})^2+(\|\bar{\nabla}w\|_2^{(\varepsilon)})^2(t)
\leq 
4M_1^2\delta^2(t)
+ \bar{C}\delta^3(t)&+|\lambda|^{-1}V^2\delta^2(t)\\&+\hat{C}T\delta^2(t)\nonumber
\end{align*}
for all $t\in [t_0+(k-1)T,t_0+kT]$.\\

Now we choose $M_2= \max\{2|\lambda|^{-\frac{1}{2}}V,8M_1\}$. Then we have
\begin{align*}
2(\|w_t\|_{L^2([t_0,t ];L^2(M))}^{(\varepsilon)})^2+(\|\bar{\nabla}w\|_2^{(\varepsilon)})^2(t)
\leq M_2^2\delta^2(t)
\end{align*}
for all $t\in [t_0,t_0+kT]$ provided $T\leq  \frac{M_2^2}{4\hat{C}}$ and $t_0+kT\leq \frac{1}{-\lambda}\log(\frac{M_2^2}{4\bar{C}})$. We use $T_1$ and $N_2$ in what follows to denote $\frac{M_2^2}{4\hat{C}}$ and $\frac{1}{-\lambda}\log(\frac{M_2^2}{4\bar{C}})$ respectively.\\

We still need to define $M_3$ and $M_4$. We consider the equation of $w_t$ which can be written as
\begin{align*}
\frac{\partial}{\partial t}(w_t)_{kk}=g^{\alpha\alpha}\bar{\nabla}_\alpha\bar{\nabla}_\alpha (w_t)_{kk}&+2\bar{R}^{i\mbox{ }i}_{\mbox{ }k\mbox{ }k}(w_t)_{ii}\\
&+ 2\delta_{1k}((g^{ll})^2\bar{\nabla}_1w_{ll}\bar{\nabla}_1(w_t)_{ll}+ c^{ii}_{\mbox{ }\mbox{ }1}\bar{\nabla}_1(w_t)_{ii})\\
&+ O(\delta)(w_t)_{kk}+ K+F_t
\end{align*}
where $\|K\|_2^{(\varepsilon)}(t)\leq \bar{C}\delta^2(t)$ for all $t\in [t_0,t_0+(k-1)T]$ and $\|F_t\|_2^{(\varepsilon)}\leq \bar{V}\delta^2(t)+\bar{V}\varepsilon$ with $\bar{V}=(1-2\lambda)V$. We can get
\begin{align*}
\frac{\partial}{\partial t}(\|w_t\|^{(\varepsilon)}_2)^2
\leq -2\lambda(\|w_t\|^{(\varepsilon)}_2)^2&+(\bar{C}+3M_3\bar{V})\delta^3(t)+M_3\bar{V}\varepsilon\delta(t)
\end{align*}
for all $t\in[t_0,t_0+(k-1)T]$.\\

 For $t\in[t_0+(k-1)T,t_0+kT]$, using the energy estimate again
\begin{align}
\|w_t\|^{(\varepsilon)}_2(t)+\|\bar{\nabla}w_t\|^{(\varepsilon)}_2(t)\leq (4C_0\bar{V}T^{\frac{1}{2}}+C_0(M_3+M_4))\delta(t),
\end{align} 
 we have 
\begin{align*}
\frac{\partial}{\partial t}(\|w_t\|^{(\varepsilon)}_2)^2
\leq &-2\lambda(\|w_t\|^{(\varepsilon)}_2)^2\\&+(\bar{C}+3(4C_0\bar{V}T^{\frac{1}{2}}+C_0(M_3+M_4))\bar{V})\delta^3(t)+M_3\bar{V}\varepsilon\delta(t).
\end{align*}

 Therefore if we set $M_3= -2\lambda \bar{V}$ and $B=\bar{C}+3M_3\bar{V}+3(4C_0\bar{V}T^{\frac{1}{2}}+C_0(M_3+M_4))\bar{V}$, 
then we can use Grownwall's inequality to get
\[
(\|w_t\|^{(\varepsilon)}_2)^2 \leq M^2_3\delta^2(t)
\]
for all $t\in[t_0,t_0+kT]$ provided $t_0+kT<\frac{1}{-\lambda}\log(\frac{-\lambda M_3^2}{4B})$. We use $N_3$ in what follows to denote $\frac{1}{-\lambda}\log(\frac{-\lambda M_3^2}{4B})$.\\

Finally, we consider the integration
\begin{align*}
(\|w_{tt}\|_2^{(\varepsilon)})^2(t)=\int_M\langle w_{tt},w_{tt}\rangle \leq& \int_M\langle g^{\alpha\alpha}\bar{\nabla}_{\alpha}\bar{\nabla}_{\alpha}w_t,w_{tt}\rangle +\int_M 2\bar{R}^{ikik}(w_t)_{ii}(w_{tt})_{kk} \\
&+\bar{C}M_4\delta^4(t)+\frac{1}{2}(\|w_t\|_2^{(\varepsilon)})^2(t)\\
&+ 4(\|F_t\|_2^{(\varepsilon)})^2(t).
\end{align*}
for $t\in[t_0,t_0+(k-1)T]$. By using the integration by parts and then integrating from $t=t_0$ to $t=(k-1)T$ we get
\begin{align*}
\frac{1}{2}(\|w_{tt}\|_{L^2([t_0,t_0+(k-1)T];L^2(M))}^{(\varepsilon)})^2\leq& -(\|\bar{\nabla}w\|_2^{(\varepsilon)})^2(t_0+(k-1)T)+(\|\bar{\nabla}w_t\|_2^{(\varepsilon)})^2(t_0)\\
&+\int_M \bar{R}^{ikik}w_{ii}w_{kk}(t_0+(k-1)T)\\
&+\int_{t_0}^{t_0+(k-1)T}\bar{C}M_4\delta^4(s)ds\\
&+ 3\bar{V}^2\delta^3(t_0+(k-1)T)+|\lambda|^{-1}\bar{V}^2\delta^2(t_0+(k-1)T).
\end{align*}
So we have
\begin{align}
&(\|w_{tt}\|_{L^2([t_0,t_0+(k-1)T];L^2(M))}^{(\varepsilon)})^2+(\|\bar{\nabla}w_t\|_2^{(\varepsilon)})^2(t_0+(k-1)T)\\
\leq &
4M_3^2\delta^2(t_0+(k-1)T)
+ \bar{D}\delta^3(t_0+(k-1)T)+|\lambda|^{-1}\bar{V}^2\delta^2(t_0+(k-1)T)\nonumber.
\end{align}
for all $t\in [t_0,t_0+(k-1)T]$.\\

 For $t\in[t_0+(k-1)T,t_0+kT]$, if we replace (5.12) by (5.20) and follow the computation to derive (5.21), we have
\begin{align*}
&(\|w_{tt}\|_{L^2([t_0+(k-1)T,t];L^2(M))}^{(\varepsilon)})^2+2(\|\bar{\nabla}w_t\|_2^{(\varepsilon)})^2(t)\\
\leq &2(\|\bar{\nabla}w_t\|_2^{(\varepsilon)})^2(t_0+(k-1)T)+ \int_{t_0+(k-1)T}^{t} \hat{D}\delta^2(s)ds\\
\leq &2(\|\bar{\nabla}w_t\|_2^{(\varepsilon)})^2(t_0+(k-1)T)+ \hat{D}T\delta^2(t)
\end{align*}
for all $t\in[t_0+(k-1)T,t_0+kT]$ and some $\hat{D}$ depending only on $M_3$ and $M_4$. By (5.21), we have
\begin{align*}
2(\|w_t\|_{L^2([t_0,t ];L^2(M))}^{(\varepsilon)})^2+(\|\bar{\nabla}w\|_2^{(\varepsilon)})^2(t)
\leq 
4M_3^2\delta^2(t)
+ \bar{D}\delta^3(t)&+|\lambda|^{-1}\bar{V}^2\delta^2(t)\nonumber\\
&+\hat{D}T\delta^2(t)\nonumber.
\end{align*}
Now we choose $M_4= \max\{2|\lambda|^{-\frac{1}{2}}\bar{V},8M_3 \}$, then we have
\begin{align*}
2(\|w_{tt}\|_{L^2([t_0,t ];L^2(M))}^{(\varepsilon)})^2+(\|\bar{\nabla}w_t\|_2^{(\varepsilon)})^2(t)
\leq M_4^2\delta^2(t)
\end{align*}
for all $t\in [t_0,t_0+kT]$ provided $T\leq \frac{M_4^2}{4\hat{D}}$ and $t_0+kT\leq \frac{1}{-\lambda}\log\frac{M_4}{4\bar{D}}$. We use $T_2$ and $N_4$ in what follows to denote $\frac{M_4^2}{4\hat{D}}$ and $\frac{1}{-\lambda}\log\frac{M_4}{4\bar{D}}$ respectively.\\

We set $T:=\min\{T_1,T_2\}$ and $\bar{N}$ be any constant smaller than $\min\{N_1,N_2,N_3,N_4\}$ and $\min_{i=1,2,3,4}\{\frac{1}{-\lambda}\log(M_i^{-1})\}$.\\

Therefore we have completed our proof.
\end{proof}
\begin{remark}
For the sup norm estimate of $|\bar{\nabla}^{(3)}w|$, we can follow the first derivative estimate in [2] again. Therefore the inequality (5.9) can be extended if our solution is solvable on each $[t_0+(k-1)T,t_0+kT]$ for all $k\leq \frac{1}{T}(\bar{N}-t_0)$.
\end{remark}

\subsection{Long-time existence}
Here is the trick we use to prove the long-time existence: first we have the short time existence of our solution on $[t_0,t_0+T]$. We apply lemma 5.16 and proposition 4.13 inductively on $[t_0+kT,t_0+(k+1)T]$. Then we use the argument of short-time existence to extend our solution.
\begin{theorem} Let $\iota$ be the constant given by theorem 5.12. There exists a universal $N>0$ such that for any $\varepsilon\leq\iota$, the Ricci-de Turck equation
\begin{align}
\frac{\partial}{\partial t}g_{ij}&=-2R_{ij}+\nabla_i V_j^{(\varepsilon)}+\nabla_j V_i^{(\varepsilon)}\\
g(x,0)&=(g_0+\varepsilon h)\nonumber
\end{align}
with $V_i^{(\varepsilon)}=g_{ik}g^{pl}(\Gamma^k_{pl}-\Gamma_{pl}^{k(\varepsilon)})$ is solvable for $t\in [t_0,N]$ where $t_0:=\frac{1}{-\lambda}\log(\varepsilon)$.
\end{theorem}
\begin{proof} Since we can write $g=g_0+\delta(t)h+w$ with $w$ satisfies equation (5.2), we only need to prove $w$ exists for $t\in[t_0,N]$.\\

 By proposition 4.13, if we set $k_0=3|Rm(g_0)|$, we can apply theorem 5.15 and remark 4.12 to obtain 
\begin{align*}
|Rm(g)|&\leq |Rm(g_0)|+|Rm(g_0+\delta(t) h)-Rm(g_0)|+|Rm(g)-Rm(g_0+\delta(t) h)|\\
&\leq|Rm(g_0)|+C\delta(t)
\end{align*}
for some constant $C$.
This implies that we can define $A:=\min\{\frac{1}{-\lambda}\log(\bar{N}),\frac{|Rm(g_0)|}{2C}\}$, such that
\[
|Rm(g)|\leq \frac{1}{2}k_0
\] 
provided $\delta(t) \leq A$.\\

Let $T$ be the constant given by lemma 5.16. By theorem 4.8, if we have
\begin{align*}
t_0+T &\leq  \frac{1}{-\lambda}\log(A):=N,
\end{align*}
 we can solve the Ricci flow $\hat{g}$ satisfies
\begin{align*}
&\frac{\partial\hat{g}}{\partial t}=-2Ric({\hat{g}});\\
&\hat{g}(t_1)=g(t_1).
\end{align*}
on $(x,t)\in M\times [t_1,t_1+T]$ with $t_1:=t_0+T$. Now we can extend the Ricci-de Turck flow on $[t_1,t_1+T]$, too. Here we need to be careful: the Ricci-de Turck flow we want to extend is the one starts from $t_0$, not the Ricci-de Turck flow starts from $t_1$. So our diffeomorphism will be
\begin{align*}
&\frac{\partial x^\alpha}{\partial t} =\frac{\partial x^\alpha}{\partial y^k}\hat{g}^{jl}(\Gamma_{jl}^k-\Gamma_{jl}^{k(\varepsilon)});\\
&x^\alpha(y,t_1)=y^\alpha
\end{align*}
which is solvable on the same interval of short time existence of the Ricci flow (Here we need the boundedness of $|\hat{g}|$ and $|\nabla \hat{g}|$, which have been proved in [1]). This implies that we can extend the Ricci-de Turck flow for a period of time $T$. By lemma 5.16, we can prove inductively that $g$ is solvable on $[t_0+(k-1)T,t_0+kT]$ for all $k\leq \frac{1}{T}(N-t_0)$.\\

Hence, we have proved the long time existence theorem.
\end{proof}
\begin{cor} Let $\iota$ be the constant given by theorem 5.12 and $N$ be the constant given by theorem 5.18. There is an universal constant $C>0$ such that if $w$ is the solution of equation (5.6) with $\varepsilon\leq \iota$, then
\begin{align}
\|w\|^{(0)}_{2}(t), \|\bar{\nabla}w\|^{(0)}_{2}(t) \leq C \delta(t);\\
|w|_0(t), |\bar{\nabla}w|_0 (t), |\bar{\nabla}\bar{\nabla}w|_0(t) \mbox{ and }|\bar{\nabla}^{(3)}w|_0(t)\leq C\delta(t)
\end{align}
for all $t\in[t_0,N]$.
\end{cor}
Recall our definition of estimate cones. We have
\begin{cor}
Let $\iota$ be the constant given by theorem 5.12 and $N$ be the constant given by theorem 5.18. There is an universal constant $M>0$ such that if $g$ is the solution of equation (5.22) with $\varepsilon\leq \iota$, then $g\in C_{h,M}(g_0)$.
\end{cor}
\begin{proof} By corollary 5.19, we can choose this cone with its opening depending on $C$.
\end{proof}

\section{Existence of Ancient Solutions}
Finally, we can prove the existence of the ancient solution. We start with finding the ancient solution of the Ricci-de Turck flow. Then we prove the existence of the de Turck diffeomorphisms.
\subsection{Existence of ancient Ricci-de Turck flow}
\begin{theorem}
Let $(M,\dot{g})$ be the Euclidean Schwarzschild manifold. Then there exists a constant $N>0$ such that the Ricci-de Turck equation
\begin{align*}
&\frac{\partial}{\partial t}g_{ij}(x,t)=-2R_{ij}(x,t)+\nabla_iV_j+\nabla_jV_i\mbox{  }x\in M;\\
&g_{ij}(x,t)\rightarrow \dot{g}(x) \mbox{ uniformly as }t\rightarrow -\infty
\end{align*}
has a solution on $(-\infty, N]$, where $V_i:=g_{ik}g^{pl}(\Gamma^k_{pl}-\dot{\Gamma}^k_{pl})$.
\end{theorem}
\begin{proof}
Let $\{\varepsilon_n\}$ be a decreasing sequence which tends to 0. For each $n\in \mathbb{N}$, we have a corresponding solution $g^{(\varepsilon_n)}$ which satisfies
\begin{align*}
\frac{\partial}{\partial t}g_{ij}&=-2R_{ij}+\nabla_i V_j^{(\varepsilon_n)}+\nabla_j V_i^{(\varepsilon_n)}\\
g(t_n)&=(g_0+\varepsilon_n h); \mbox{ }t_n:=\frac{\log \varepsilon_n}{-\lambda}
\end{align*}
where $V_i^{(\varepsilon_n)}=g_{ik}g^{pl}(\Gamma^k_{pl}-\Gamma_{pl}^{k(\varepsilon_n)})$.\\

We can rewrite the equation $\frac{\partial}{\partial t}g_{ij}=-2R_{ij}+\nabla_i V_j^{(\varepsilon_n)}+\nabla_j V_i^{(\varepsilon_n)}$ as
\[
\frac{\partial}{\partial t} g_{ij}= \bar{\nabla}_\alpha g^{\alpha\alpha}\bar{\nabla}_\alpha g_{ij} + B_{ij}(\bar{\nabla}g, g) \bar{\nabla} g_{ij}+ C_{ij}(\bar{\nabla}g, g).
\]
for some smooth functions $B_{ij}$ and $C_{ij}$. Therefore we can regarded $g^{(\varepsilon_n)}$ as a solution of
\begin{align*}
\frac{\partial}{\partial t} g_{ij}&= \bar{\nabla}_\alpha(g^{(\varepsilon_n)})^{\alpha\alpha}\bar{\nabla}_\alpha g_{ij} + B_{ij}(\bar{\nabla}g^{(\varepsilon_n)}, g^{(\varepsilon_n)}) \bar{\nabla} g_{ij}+ C_{ij}(\bar{\nabla}g^{(\varepsilon_n)}, g^{(\varepsilon_n)})\\
g(t_n)&=(g_0+\varepsilon_n h); \mbox{ }t_n:=\frac{\log \varepsilon_n}{-\lambda}.
\end{align*}

 Now we choose a sequence of compact subset $D_i\subset M$ with $\cup_{i\in \mathbb{N}} D_i=M$. We then have a sequence of subsets in the space-time $\{P_i:=D_i\times(t_i, N]\subset M\times (-\infty,N]| i\in \mathbb{N}\}$. By corollary 4.4 and corollary 5.19, we can find a subsequence $\{g^{(\varepsilon_{n_l})}\}$ such that $g^{(\varepsilon_{n_l})}$ and $\bar{\nabla}g^{(\varepsilon_{n_l})}$ converge uniformly on $P_1$. By standard diagonal process, we have a subsequence, say $\{g^{(\varepsilon_{n_l})}\}$ again, such that $g^{(\varepsilon_{n_l})}$ and $\bar{\nabla}g^{(\varepsilon_{n_l})}$ converge uniformly on $P_i$ for each $i\in \mathbb{N}$.\\

 Let $g^{(\varepsilon_{n_l})}\rightarrow g^{(\infty)}$ as $l\rightarrow \infty$. Therefore $B_{ij}(\bar{\nabla}g^{(\varepsilon_{n_l})}, g^{(\varepsilon_{n_l})})$ and $C_{ij}(\bar{\nabla}g^{(\varepsilon_{n_l})}, g^{(\varepsilon_{n_l})})$ will converge uniformly to $B_{ij}(\bar{\nabla}g^{(\infty)}, g^{(\infty)})$ and $C_{ij}(\bar{\nabla}g^{(\infty)}, g^{(\infty)})$ respectively on $P_i$ for each $i\in \mathbb{N}$.\\

  Moreover, by corollary 5.20, all $g^{(\varepsilon_n)}$ belong to the estimate cone $C_{h,M}(g_0)$. Combining with inequality (5.24), we have that $g^{(\varepsilon_n)}$ are bounded $L^2(P_1)$. Therefore we will have a convergent subsequence converge weakly in $P_1$. By standard diagonal process, we will get a subsequence
\[
g^{(\varepsilon_{n_l})}\rightharpoonup g^{(\infty)},
\]
on $P_k$ for all $k$. By the definition of weak limit, we conclude that $g^{(\infty)}$ is a weak solution of equation
\[
\frac{\partial}{\partial t} g_{ij}= \bar{\nabla}_\alpha(g^{(\infty)})^{\alpha\alpha}\bar{\nabla}_\alpha g_{ij} + B_{ij}(\bar{\nabla}g^{(\infty)}, g^{(\infty)}) \bar{\nabla} g_{ij}+ C_{ij}(\bar{\nabla}g^{(\infty)}, g^{(\infty)}).
\]
By the regularity theorem of parabolic PDE's, it is an ancient Ricci-de Turck flow, too.\\

Finally, we should prove the this solution is not a trivial, i.e. $g\neq \dot{g}$.\\

We write $g^{(\varepsilon_n)}=\dot{g}+\delta(t)h+w^{(\varepsilon_n)}$. Using remark 5.13,  $w^{(\varepsilon_n)}_{kk}$ will satisfy a equation of the form $\frac{\partial}{\partial t}w^{(\varepsilon_n)}_{kk}=L(w^{(\varepsilon_n)}_{kk})+O(\delta^2)$ with some elliptic operator $L$ and for all $t\in [t_n, N]$ and $k\neq 1$. We change the coordinate by taking $s$ as defined by (5.4). Then $L$ will be uniformly parabolic. In this case, we can fix this coordinate and applying the energy estimate of single equation to get 
\[
|w^{(\varepsilon_n)}_{kk}|\leq C\dot{g}_{kk}\delta^2(t)
\]
for all $t\in [t_n,N]$ and $k\neq 1$. So we have $\dot{g}_{kk}+\delta h_{kk} + w^{(\varepsilon_n)}_{kk}=\dot{g}_{kk}+\delta h_{kk}+O(\delta^2)$ for all $k\neq 1$ which is not equal to $\dot{g}_{kk}$ for all $t$.\\

Therefore we have completed this proof.
\end{proof}

\subsection{Solvablity of de Turck diffeomorphisms} Finally, we need to show that the 1st order PDE of the de Turck deffeomorphism can be solve. Recall that in section 3, we define the following PDE
\begin{align*}
&\frac{\partial y^\alpha}{\partial t}=\frac{\partial y^\alpha}{\partial x^k}g^{jl}(\Gamma^k_{jl}-\dot{\Gamma}^k_{jl})\\
&y^\alpha(x,-\infty)=x^\alpha
\end{align*}
Since our metric is radially symmetric, our equation can be reduced to a single equation
\begin{align*}
&\frac{\partial y^1}{\partial t}=\frac{\partial y^1}{\partial x^1}g^{jj}(\Gamma^1_{jj}-\dot{\Gamma}^1_{jj})\\
&y^1(x,-\infty)=x^1
\end{align*}
If we change our variable by defining $\delta(t)$ to be $e^{-\lambda t}$, we will have a new equation
\begin{align*}
&(-\lambda \delta)\frac{\partial y^1}{\partial \delta}=\frac{\partial y^1}{\partial x^1}g^{jj}(\Gamma^1_{jj}-\dot{\Gamma}^1_{jj})\\
&y^1(x,0)=x^1,
\end{align*}
or we can say
\begin{align}
&\frac{\partial y^1}{\partial \delta}=[\frac{1}{-\lambda \delta}g^{jj}(\Gamma^1_{jj}-\dot{\Gamma}^1_{jj})]\frac{\partial y^1}{\partial x^1}\\
&y^1(x,0)=x^1.
\end{align}

Now, by corollary 4.4 and corollary 5.19, we have $|\frac{1}{\delta}g^{jj}(\Gamma^1_{jj}-\dot{\Gamma}^1_{jj})|$ and its derivative are uniformly bounded on $\delta\in (0,e^{-\lambda N}]$. Therefore by standard first order PDE theory, we can find the characteristic curves and solve this equation of $\delta\in[0, e^{-\lambda N}]$.\\

Therefore, we finish our argument and prove theorem 1.1.

\section{Appendix}
Here we follow the notations of section 5.4 and prove lemma 5.9 

\subsection{ Higher derivative estimates} Here we prove the higher derivative estimates of $g$. This estimate is essential for the proof of lemma 5.9.\\

Let $g$ be the solution of the Ricci-de Turck equation
\begin{align}
&\frac{\partial}{\partial t}g_{ij}(x,t)=-2R_{ij}(x,t)+\nabla_iV_j+\nabla_jV_i\mbox{  }\mbox{ for all }(x,t)\in M\times[0,T];\\
&g_{ij}(x,0)=\bar{g}(x) \mbox{ }\mbox{ for all }x\in M\nonumber
\end{align}
with $|Rm(\bar{g})|\leq k_0$ for some $T>0$. By using de Turck method, we can find a family of diffeomorphisms $\varphi_{t}:M\rightarrow M$ such that $\hat{g}=(\varphi_t^{-1})^*(\bar{g})$ is a solution of Ricci flow equation.\\

In [1], Shi proved the following higher derivative estimate: for any $m\in \mathbb{N}$, there exist
constants $C_m$ depending on $m$ and $k_0$ such that
\begin{align}
|\hat{\nabla}^{(m)} Rm(\hat{g}) |\leq \frac{C_m}{t^{\frac{m}{2}}}.
\end{align}
In fact, Shi proved the following local estimate:
\begin{lemma}
Let $g$ be a Ricci flow defined on $M\times [0,T]$ for some $T>0$. There exist constants $\theta$ and $C_l$, $l\in \mathbb{N}$ depending only on the dimension of $M$ such that if 
\begin{align*}
|Rm|\leq K \mbox{ on } B_r(p)\times[0,\frac{\theta}{K}]
\end{align*}
for some $K\geq\frac{\theta}{T}$ where $B_r(p)$ is the geodesic ball centered at $p$ with radius $r$ with respect to the metric at $t=0$, then we have
\begin{align}
|\nabla^l Rm(p,t)|^2\leq C_lK^2(\frac{1}{r^{2l}}+\frac{1}{t^l}+K^l)
\end{align} 
for all $l\in\mathbb{N}$ and $(p,t)\times[0,\frac{\theta}{K}]$.
\end{lemma}

Reader can see theorem 1.4.2 in [12] for the proof of this lemma.\\

Base on this lemma, we can prove the following proposition for the Ricci-de Turck flow $g$.
\begin{pro} Let $g$ be the Ricci-de Turck flow given by theorem 4.9. Then for any $m\in \mathbb{N}$, there exists a constant $C_m$ depending on $m$, $T$ and $k_0$ such that
\[
\sup_{x\in M}|\bar{\nabla}^{(m)}g|(x,t)\leq \frac{C_m}{t^{\frac{m-1}{2}}}
\]
for all $t\in[0,T]$.
\end{pro}
\begin{proof} First of all, we define the tensor $A^k_{ij}=\Gamma^k_{ij}-\bar{\Gamma}^k_{ij}=\frac{1}{2}g^{kl}(\bar{\nabla}_ig_{jl}+\bar{\nabla}_jg_{il}-\bar{\nabla}_lg_{ij})$. So we have
\[
|A|\leq 2|g^{-1}||\bar{\nabla}g|.
\]
By theorem 5.4, we have $|g^{-1}|$ is bounded. So we have
\begin{align}
|\bar{\nabla}^{(k)}A|\leq C\sum_{l=0}^{k+1}|\bar{\nabla}^{(l)}g||\bar{\nabla}^{(k-l+1)}g| \mbox{ }\mbox{ for all } k\in \mathbb{N}
\end{align}
and a combinatorial constant $C$ depending on $k$.\\

Secondly, by (7.2), since $|\hat{\nabla}^{(m)} Rm(\hat{g}) |\leq C_m t^{-\frac{m}{2}}$, we have
\begin{align}
\varphi_t^*(|\hat{\nabla}^{(m)} Rm(\hat{g}) |)= |\nabla^{(m)} Rm|\leq \frac{C_m}{t^{\frac{m}{2}}}.
\end{align}

In what follows, $N_1$ and $I_1$ denote constants that depending only on $T$ and $K_0$. Inductively, for any $p\in\mathbb{N}$, $N_p$ denotes a constant that depending only on $N_1,\cdots, N_{p-1}$, $I_1,\cdots, I_{p-1}$, $p$, $T$, $k_0$ and $I_p$ denotes a constant that depending only on $N_1,\cdots, N_{p}$, $p$, $T$ and $k_0$; their precise values can be assumed to increase between each successive appearance.\\

Now we prove the result inductively. We claim that
\begin{align}
|\nabla^{(p)} V|\leq \frac{N_p}{t^{\frac{p}{2}}} \mbox{ and } |\bar{\nabla}^{(p+1)} g|\leq \frac{N_p}{t^{p/2}}
\mbox{ }\mbox{ for all } p\in \mathbb{N}.
\end{align}

We start with the case $p=1$. By theorem 5.5, we have $|\bar{\nabla}g|$ and $|\frac{\partial}{\partial t}g|$ are bounded. Since $g$ is radially symmetric, we have $|\bar{\nabla}\bar{\nabla}g|$ is bounded. Meanwhile, because $V_i=g_{ik}g^{pl}A^k_{pl}$, we have
\[
|\nabla V| = |(\bar{\nabla}+A)V|\leq C(|A|^2+|\bar{\nabla}\bar{\nabla}g|+|\bar{\nabla}g|^2)
\]
which is bounded. So the $p=1$ case is ture.\\

Suppose inequality (7.6) is true for $p\leq k$. We consider the evolution equation of $V$, which can be written as
\begin{align*}
\frac{\partial}{\partial t} V=\Delta V &+ g^{-2}\ast Rm \ast V + g^{-1}\ast Ric\ast V\\ &+ g^{-1}\ast \nabla V\ast V +g^{-1}\ast \bar{\nabla}g\ast Ric\ast g+ g^{-1}\bar{\nabla}g\ast \bar{\nabla} V\ast g
\end{align*}
where $\ast$ is the product of tensors. Now by Cauchy's inequality, we have
\begin{align}
\frac{\partial}{\partial t}|\nabla^{(k)} V|\leq \Delta |\nabla^{(k)} V|^2&-|\nabla^{(k+1)}V|\nonumber\\&+C\sum_{l=0}^{k}(\sum_{i_1+i_2+...+i_{l}=l}|\nabla^{(i_1)} g||\nabla^{(i_2)}g|\cdots|\nabla^{(i_l)}Ric|)^2\\
&+C(\sum_{j=0}^{k}|\nabla^{(j)}V|^2
+|\nabla^{(k+1)}g|^2).\nonumber
\end{align}

 Since $\nabla=\bar{\nabla}+A$, we have
\begin{align*}
|\nabla^{(p)} g|\leq C\sum_{j_1+...+j_{p-1}+l=p; l\geq 1}|\bar{\nabla}^{(j_1)}A||\bar{\nabla}^{(j_2)}A|\cdots |\bar{\nabla}^{(j_{p-1})}A||\bar{\nabla}^{(l)}g|.
\end{align*}
for all $p\leq k+1$. Now by (7.4) and the induction hypothesis, we have 
\[
|\bar{\nabla}^{(j_{p-1})}A|\leq N_{j_{p-1}+1}\frac{1}{t^{\frac{j_{p-1}}{2}}}.
\]
So we have
\[
|\nabla^{(p)} g|\leq \frac{I_{p}}{t^{\frac{p-1}{2}}}.
\]
for all $p\leq k$ and $|\nabla^{(k+1)} g|\leq N_{k+1}$\\

 Now we use (7.5) and the induction hypothesis, equation (7.7) will become
\begin{align}
\frac{\partial}{\partial t}|\nabla^{(k)} V|^2\leq \Delta |\nabla^{(k)} V|^2-|\nabla^{(k+1)}V|+\frac{N_{k+1}}{t^k}.
\end{align}

Here we follow the technique in the proof of Shi's derivative estimate in [12]. If we define
\[
W_k=[M_k\frac{1}{t^k}+|\nabla^{(k)}V|^2]|\nabla^{(k+1)}V|^2
\]
for some $M_k$ large enough (depending on $N_{k+1}$), by (7.8) and the Cauchy inequality, we have
\[
\frac{\partial}{\partial}W_k\leq \Delta W_k+\frac{W^2_k}{CM^2_k t^{-2k}}+CM_k^2 t^{2k+1}
\]
for some universal constant $C$. Let $F_k=bW_k t^k$. By the Cauchy inequality, we have
\[
\frac{\partial}{\partial t}F_k\leq \Delta F_k-t^kF_k^2+t^{-(k+2)}
\]
by taking $b$ small enough.\\

Now we use the barrier function constructed in theorem 1.4.2. in [12]. In [12], Zhu and Cao prove that for any $k\in\mathbb{N}$, $x\in M$ there exists a function $H_k$ defined on $B_1(x)\times [0,T]$ which satisfies\\
\ \\
1. $H_k$ increases to $+\infty$ on $(\partial B_1(x)\times [0,T])\cup B_1(x)\times \{0\}$,\\
2. $H_k\leq \alpha_k(t^{-(k+1)}+\beta_k)$ on $B_\frac{1}{2}(x)\times [0,T]$ for some constants $\alpha_k$, $\beta_k$,\\
3.
\[
\frac{\partial}{\partial t}H_k > \Delta H_k-u^{-k}H^2_k+u^{k+2}
\] 
\ \ \ \ where $u:=r^{-2}+t^{-1}$.\\

 Since $u>t^{-1}$, we have
\[
\frac{\partial}{\partial t}F_k\leq \Delta F_k-u^{-k}F_k^2+u^{k+2}
\]
By the maximum principle, we have
\[
F_k\leq H_k
\]
on $B_1(x)\times [0,T]$ which implies
\[
|\nabla^{(k+1)}V|^2(x,t)\leq b^{-1}(M_k+N_k)^{-1}t^k[\alpha_k(t^{-(k+1)}+\beta_k)]^2\leq \frac{N_{k+1}}{t^{k+1}}.
\]

We still need to prove that
\[
|\bar{\nabla}^{(k+2)}g|\leq \frac{N_{k+1}}{t^{\frac{k+1}{2}}}.
\]
By our equation, this estimate can be obtained by proving the inequality
\begin{align}
|-2\bar{\nabla}^{(k)} Ric+2\bar{\nabla}^{(k)}\nabla V|\leq \frac{N_{k+1}}{t^{\frac{k+1}{2}}}.
\end{align}
Since we have
\begin{align*}
&|\bar{\nabla}^{(k)}Ric|\leq |(\nabla-A)^{(k)}Ric|\\ \leq& C\sum_{j_1+...+j_{k-1}+l=k; l\geq 1}|\nabla^{(j_1)}A||\nabla^{(j_2)}A|\cdot\cdot |\nabla^{(j_{k-1})}A||\nabla^{(l)}Ric|\\
\leq& C\sum_{j_1+...+j_{k-1}+l\leq k; l\geq 1}|\bar{\nabla}^{(j_1)}A||\bar{\nabla}^{(j_2)}A|\cdot\cdot |\bar{\nabla}^{(j_{k-1})}A||\nabla^{(l)}Ric|\\
\leq& \frac{N_{k}}{t^{\frac{k}{2}}}
\end{align*}
and
\begin{align*}
|\bar{\nabla}^{(k)}\nabla V|=|(\nabla-A)^{(k)}\nabla V|\leq \frac{N_{k+1}}{t^{\frac{k+1}{2}}},
\end{align*}
we get (7.9) immediately.
\end{proof}

\subsection{ Proof of lemma 5.9} We use the $p$-coordinate in what follows. We fix a $t\in (t_0,t_0+T_0]$ and prove that
\begin{align}
|\bar{\nabla}^kv|(p)\rightarrow 0
\end{align}
as $p\rightarrow 1$.\\

By proposition 7.2, we have $|\bar{\nabla}^kv|$ are bounded for all $k$. Here we prove (7.10) inductively.\\

When $k=1$, we have
\[
2|\bar{\nabla}_iv|^2=\bar{\nabla}_i^2|v|^2-2\langle \bar{\nabla}_i^2v,v\rangle
\]
Since $|\langle \bar{\nabla}_i^2v,v\rangle|\leq |\bar{\nabla}_i^2v||v|$ which tends to zero by proposition 5.7 and boundedness of $ |\bar{\nabla}_i^2v|$, we should prove that $\bar{\nabla}_i^2|v|^2(p)\rightarrow 0$ as $p\rightarrow 1$. Moreover, since $|v|$ is radially symmetric, we only need to prove that $\bar{\nabla}_1^2|v|^2=(\frac{\partial}{\partial p})^2|v|^2$ tends to zero as $p\rightarrow 1$.\\

 Recall that for all $l\in\mathbb{N}$, $|v|\leq C_l (1-p)^l$ for some $C_l$. Therefore we have the difference equation
\begin{align*}
\frac{|v|^2(1-2h)-2|v|^2(1-h)+|v|^2(1)}{h^2}=(\frac{\partial}{\partial p})^2|v|^2(\xi)
\end{align*}
for some $\xi\in (1-2h,1)$ by the mean value theorem. The left hand side
\[
|\frac{|v|^2(1-2h)-2|v|^2(1-h)+|v|^2(1)}{h^2}|\leq C_2^2\frac{(2h)^4+h^4}{h^2}\rightarrow 0
\]
as $h\rightarrow 0$. So we can find a sequence $\{\xi_n|\xi\rightarrow 1\}$ such that
\[
\frac{\partial}{\partial p}^2|v|^2(\xi_n)\rightarrow 0.
\]

Now, since $|\bar{\nabla}^kv|$ is bounded for all $k$, we have $|\frac{\partial}{\partial p}^3|v|^2|$ is bounded by a constant, say $A$. For any $\varepsilon>0$, we can always find $\xi_n$ such that $|\xi_n-1|\leq \frac{\varepsilon}{2A}$ and $|\frac{\partial}{\partial p}^2|v|^2(\xi_n)|\leq \frac{\varepsilon}{2}$. Using the mean value theorem, we get for all $p\in(\xi_n, 1]$ there exists $\eta\in (\xi_n,p)$ such that
\[
|(\frac{\partial}{\partial p})^2|v|^2(p)|\leq |(\frac{\partial}{\partial p})^2|v|^2(\xi_n)|+|(\frac{\partial}{\partial p})^3|v|^2(\eta)||\xi_n-p|\leq \varepsilon.
\]
So we have $\frac{\partial}{\partial p}^2|v|^2(p)\rightarrow 0$ as $p\rightarrow 1$.\\

Suppose that we have $|\bar{\nabla}^lv|(p)\rightarrow 0$ as $p\rightarrow 1$ for all $l\leq k$. We consider the following equality
\[
2C^{2k}_k|\bar{\nabla}_i^{k+1} v|^2=\bar{\nabla}_i^{k+1}|v|^2-2\sum_{l=0}^{k-1}C^{2k}_l\langle \bar{\nabla}^{2k-l}v,\bar{\nabla}^lv\rangle.
\]
Since $|\langle \bar{\nabla}^{2k-l}v,\bar{\nabla}^lv\rangle|\leq |\bar{\nabla}^{2k-l}v||\bar{\nabla}^lv|$ tends to zero as $p\rightarrow 1$, we only need to show
\[
(\frac{\partial}{\partial p})^{k+1}|v|^2(p)\rightarrow 0
\]
as $p\rightarrow 1$. Again, by using the difference equation, we can find a sequence $\{\xi_n|\xi\rightarrow 1\}$ such that
\[
(\frac{\partial}{\partial p})^{k+1}|v|^2(\xi_n)\rightarrow 0.
\]
as $n\rightarrow \infty$. Now using the mean value theorem and boundedness of $(\frac{\partial}{\partial p})^{k+2}|v|^2$ we prove (7.10).\\

By (7.10), we can extend $|v|^2$ as a smooth function defined on $[0,1]$. Since $|v|^2\leq C_l(1-p)^l$ for all $l\in \mathbb{N}$, we can prove that for any $l\in\mathbb{N}$, there exists a constant $C'_l$ such that $|\bar{\nabla}^{(k)}v|\sqrt{\bar{g}}\leq C'_l(1-p)^l$.\\

\bibliographystyle{amsplain}

\end{document}